\documentclass[11pt,a4paper]{article}

\usepackage{graphicx}
\usepackage[a4paper, hmargin={2.7cm,2.7cm},vmargin={3.3cm,3.3cm}]{geometry}

\usepackage[T1]{fontenc}
\usepackage[utf8]{inputenc}
\usepackage[english]{babel}

\usepackage{amsmath,amssymb,amsfonts}
\usepackage{amsthm,latexsym}
\usepackage{etoolbox}
\usepackage{color}

\usepackage[pagewise,displaymath, mathlines]{lineno}

\usepackage{cite}
\usepackage{paralist}
\usepackage{url}
 \usepackage[bf,compact,pagestyles,small]{titlesec}
 \titlespacing*{\section}{0pt}{14pt}{4pt}
 \titlespacing*{\subsection}{0pt}{8pt}{3pt}

\makeatletter
\patchcmd{\ttlh@hang}{\parindent\z@}{\parindent\z@\leavevmode}{}{}
\patchcmd{\ttlh@hang}{\noindent}{}{}{}
\makeatother

\usepackage{fancyhdr,lastpage}
\def\maketimestamp{\count255=\time
\divide\count255 by 60\relax
\edef\thetime{\the\count255:}%
\multiply\count255 by-60\relax
\advance\count255 by\time
\edef\thetime{\thetime\ifnum\count255<10 0\fi\the\count255}
\edef\thedate{\number\day-\ifcase\month\or Jan\or Feb\or Mar\or
             Apr\or May\or Jun\or Jul\or Aug\or Sep\or Oct\or
             Nov\or Dec\fi-\number\year}
\def\timstamp{\hbox to\hsize{\tt\hfil\thedate\hfil\thetime\hfil}}}
\maketimestamp
\fancypagestyle{plain}{%
\fancyhf{} 
\fancyfoot[C]{\small \thepage{} of \pageref{LastPage}}}

\numberwithin{equation}{section}  
\allowdisplaybreaks[4]
\newtheorem{thm}{Theorem}[section]
\newtheorem{lem}[thm]{Lemma}
\newtheorem{prop}[thm]{Proposition}
\newtheorem{cor}[thm]{Corollary}
\theoremstyle{definition}
\newtheorem{defn}[thm]{Definition} 
\newtheorem{ex}{Example}
\AtEndEnvironment{ex}{\null\hfill\ensuremath{\blacksquare}}
\theoremstyle{remark}
\newtheorem{rem}{Remark}

\DeclareMathOperator{\supp}{supp} %
\DeclareMathOperator*{\esssup}{ess\,sup} %
\DeclareMathOperator*{\essinf}{ess\,inf} %

\DeclareMathOperator{\exponential}{e}

\newcommand{\csi}[1][g]{\ensuremath{\setsmall{T_{c k} {#1}_j}_{k \in
      \Z, j \in J}}}
\newcommand{\gsi}[1][g]{\ensuremath{\setsmall{T_{c_j k} {#1}_j}_{k \in
      \Z, j \in J}}}

\newcommand{\gsiZ}[1][g]{\ensuremath{\setsmall{T_{c_j k} {#1}_j}_{k \in
      \Z, j \in \Z}}}
\newcommand{\wpj}[1][\psi]{\ensuremath{\setsmall{D_{a_j}T_{bk}E_{d_j}{#1}}_{k\in
      \Z, j\in J}}}
\newcommand{\siwp}[1][\psi]{\ensuremath{\setsmall{T_{bk}D_{a_j}E_{d_j}{#1}}_{k\in
      \Z, j\in J}}}
\newcommand{\Lloc}{L^1_{\mathrm{loc}}}
\newcommand{\myexp}[1]{\exponential^{#1}}

\newcommand{\sign}[1]{\mathrm{sgn}(#1)}

\newcommand{\tran}[1][k]{T_{#1}} 
\newcommand{\modu}[1][m]{E_{#1}} 

\newcommand{\eps}{\ensuremath{\varepsilon}}

\newcommand*{\numbersys}[1]{\ensuremath{\mathbb{#1}}}
\newcommand*{\C}{\numbersys{C}}
\newcommand*{\R}{\numbersys{R}}

\newcommand*{\Z}{\numbersys{Z}}

\newcommand*{\N}{\numbersys{N}}

\newcommand*{\cP}{\mathcal{P}}

\newcommand{\itvoc}[2]{\ensuremath{\left({#1},{#2}\right]}} 
\newcommand{\itvoo}[2]{\ensuremath{\left({#1},{#2}\right)}} %
\newcommand{\itvcc}[2]{\ensuremath{\left[{#1},{#2}\right]}} %
\newcommand{\itvco}[2]{\ensuremath{\left[{#1},{#2}\right)}} %
\newcommand{\itvccs}[2]{\ensuremath{\lbrack{#1},{#2}\rbrack}} %
\newcommand{\abs}[1]{\ensuremath{\left\lvert#1\right\rvert}}
\newcommand{\abssmall}[1]{\ensuremath{\lvert#1\rvert}}
\newcommand{\absbig}[1]{\ensuremath{\bigl\lvert#1\bigr\rvert}}

\newcommand{\norm}[2][]{\ensuremath{\left\lVert#2\right\rVert_{#1}}}

\newcommand{\normsmall}[2][]{\ensuremath{\lVert#2\rVert_{#1}}}
\newcommand{\innerprod}[3][]{\ensuremath{\left\langle #2,#3\right\rangle_{\! #1}}}

\newcommand{\set}[1]{\ensuremath{\left\lbrace{#1}\right\rbrace}}

\newcommand{\setprop}[2]{\ensuremath{\left\lbrace{#1} : {#2}\right\rbrace}}

\newcommand{\setsmall}[1]{\ensuremath{\lbrace{#1}\rbrace}}
\newcommand{\seqsmall}[1]{\ensuremath{\lbrace{#1}\rbrace}}

\newcommand{\setpropsmall}[2]{\ensuremath{\lbrace{#1} : {#2}\rbrace}}

\newcommand\cD{\mathcal{D}}
\newcommand\cE{\mathcal{E}}

\def\bpr{\begin{propp}}
\def\epr{\end{prop}}

\newcommand{\ltr}{ L^2(\mathbb R) }

\newcommand{\mn}{\mathbb N}
\newcommand{\mr}{\mathbb R}

\newcommand{\bbz}{\mathbb Z}

\newcommand{\mz}{\mathbb Z}

\newcommand{\mts}{ \{E_{mb}T_{na}g \}_{m,n \bbz}}

\newcommand{\aj}{\{a_j\}_{j\in\Z}}

\newcommand{\irw}{\{D_{a_j}T_{bk}\psi\}_{j,k\in\Z}}

\def\bqs{\begin{equation}}
\def\eqs{\tag*{$\square$}\end{equation}\par\bigskip}

\def\hp{\hat{\psi}}

\def\ga{\gamma}

\def\hp{\hat{\psi}}

\def\bop{\begin{op}\rm}
\def\eop{\end{op}}

\def\bes{\begin{eqnarray*}}
\def\ens{\end{eqnarray*}}
\def\bei{\begin{itemize}}
\def\eni{\end{itemize}}
\def\bt{\begin{thm}}
\def\et{\end{thm}}
\def\bc{\begin{cor}}
\def\ec{\end{cor}}
\def\bpr{\begin{prop}}
\def\epr{\end{prop}}
\def\bl{\begin{lemma}}
\def\el{\end{lemma}}
\def\bd{\begin{defn}}
\def\ed{\end{defn}}
\def\bex{\begin{ex}}
\def\enx{\end{ex}}
\def\bfi{\begin{fig}}
\def\efi{\end{fig}}

\def\hpi{\hat{\psi}}

\def\bbz{{\mathbb Z}^d}

\def\sumj{\sum_{j\in \Z}}
\def\sumk{\sum_{k\in \Z}}

\def\hp{\widehat{\psi}}

\def\bbz{\in\Z}
\def\L(f){\sum_{j\in J}\sumk\frac{1}{c_j}\int_{\Bbb R}\hat{f}(\gamma)\overline{\hat{f}(\gamma+c_j^{-1}k)\hat{\phi}(\gamma)}\hat{\phi}(\gamma+c_j^{-1}k)d\gamma}

\def\wps{\{D_{a^j}T_{bk}E_{d_m}\psi\}_{j\in J,m,k\bbz}}

\usepackage{hyperref}
\hypersetup{
 pdfview={FitH},
 pdfstartview={FitH},
 pdfauthor = {Ole Christensen, Marzieh Hasannasab, Jakob Lemvig}
 pdftitle = {Explicit constructions and properties of wave packet systems in $L^2(R)$},
 pdfkeywords = {Dual frames, frame, Gabor system, generalized shift
   invariant system, LIC, local integrability condition, wave packets, wavelets},
 pdfcreator = {LaTeX with hyperref package},
 pdfproducer = PDFlatex}

\makeatletter
\def\blfootnote{\xdef\@thefnmark{}\@footnotetext}
\def\subjclass{\xdef\@thefnmark{}\@footnotetext}
\long\def\symbolfootnote[#1]#2{\begingroup%
\def\thefootnote{\fnsymbol{footnote}}\footnote[#1]{#2}\endgroup}
\if@titlepage
  \renewenvironment{abstract}{%
      \titlepage
      \null\vfil
      \@beginparpenalty\@lowpenalty
      \begin{center}%
        \bfseries \abstractname
        \@endparpenalty\@M
      \end{center}}%
     {\par\vfil\null\endtitlepage}
\else
  \renewenvironment{abstract}{%
      \if@twocolumn
        \section*{\abstractname}%
      \else
        \small
        \list{}{%
          \settowidth{\labelwidth}{\textbf{\abstractname:}}
          \setlength{\leftmargin}{50pt}
          \setlength{\rightmargin}{50pt}
          \setlength{\itemindent}{\labelwidth}
          \addtolength{\itemindent}{\labelsep}
        }
        \item[\textbf{\abstractname:}]

      \fi}
      {\if@twocolumn\else\endlist\fi}
\fi
\makeatother

\begin{document}

\title{Explicit constructions and properties of generalized shift-invariant systems in $L^2(\mathbb{R})$}

\date{\today}

 \author{Ole Christensen\footnote{Technical University of Denmark, Department of Applied Mathematics
     and Computer Science, Matematiktorvet 303B, 2800 Kgs.\ Lyngby, Denmark}, 
Marzieh Hasannasab$^\ast$,
Jakob Lemvig$^\ast$
}

 \blfootnote{2010 {\it Mathematics Subject Classification.} Primary 42C15, 43A70, Secondary: 43A32.}
 \blfootnote{{\it Key words and phrases.} Dual frames, frame, Gabor system, Generalized shift
   invariant system, LIC, local integrability condition, wave packets, wavelets}
\blfootnote{{\it E-mail addresses}: \url{ochr@dtu.dk} (Ole Christensen), \url{mhas@dtu.dk} (Marzieh Hasannasab), \url{jakle@dtu.dk} (Jakob Lemvig)}

\maketitle

\thispagestyle{plain}
\begin{abstract} Generalized shift-invariant (GSI) systems, originally introduced by Hern\'andez, Labate \& Weiss and Ron \& Shen, provide a common frame work for analysis of Gabor systems, wavelet systems,
wave packet systems, and other types of structured function systems. In this paper we analyze three important aspects of such systems. First, in contrast to the known cases of Gabor frames and wavelet frames, we show that for a GSI system forming a frame, the
Calder\'on sum is not necessarily bounded by the lower frame bound. We identify a technical condition implying that the Calder\'on sum is  bounded by the lower frame bound and show that under a weak assumption  the condition is equivalent with the local integrability condition introduced by Hern\'andez et.\ al.  Second, we provide explicit and general constructions of frames
and dual pairs of frames
having the GSI-structure. In particular, the setup applies to wave packet systems and in contrast to the
constructions in the literature, these constructions are not based on characteristic functions
in the Fourier domain. Third, our results provide insight into the local integrability condition
(LIC).
\end{abstract}

\section{Introduction}

Generalized shift-invariant systems provide a common framework for analysis of
a large class of  function systems in $\ltr.$ Defining the translation operators $T_c,  c\in \R$,
by $T_cf(x)=f(x-c)$, a \emph{generalized shift-invariant (GSI) system} has the
form   $\gsi$, where
$\set{c_j}_{j \in J}$ is a countable set in
$\R_+$ and $g_j \in L^2(\R).$ GSI systems were introduced by Hern\'andez, Labate
\& Weiss~\cite{MR1916862}, and  Ron \& Shen~\cite{MR2132766}.

In the analysis of a GSI system, the function $\sum_{j\in J}
c_j^{-1} | \hat{g}_j(\cdot) |^2$, which we will call the Calder\'on sum in analogue with the
standard terminology used in the special case of a wavelet
system, plays an important role. Intuitively, the Calder\'on sum measures the total energy
concentration of the generators $g_j$ in the frequency domain. Hence,
whenever a GSI system has the frame property, one would expect the
Calder\'on sum to be bounded from below since the GSI frame can reproduce
all frequencies in a stable way.
Indeed, whenever a Gabor frame or a wavelet frame
is
considered as a GSI system in the natural way (see the details below), it is known that the Calder\'on sum is bounded
above and below by the upper and lower frame bounds, respectively.
In
the general case of a GSI system the Calder\'on sum is known to be
bounded above by the upper frame bound.  In this paper we prove by
an example that the Calder\'on sum is not always bounded below by the
lower frame bound. On the other hand, we identify a technical
condition implying that the Calder\'on sum is  bounded by the lower
frame bound. Under
a weak assumption, this condition is proved to be equivalent with the
local integrability condition introduced by Hern\'andez et.\ al.~\cite{MR1916862}.

Our second main contribution is to provide  constructions of pairs of
dual frames having the GSI structure. The construction procedure
allows for smooth and well-localized generators, and it
unifies several known constructions of dual frames with Gabor,
wavelet, and so-called Fourier-like structure \cite{MR2481501,Lemvig2012,MR2407863,MR2587577,MR3061703}. Due to its generality the setup is technical, but nevertheless it is possible to extract attractive new constructions, as we will explain below.

We will apply our results for GSI systems on the important special
case of wave packet systems. We consider necessary and sufficient
conditions for frame properties of wave packet systems.  In
particular, by the just mentioned construction procedure, we obtain
dual pairs of wave packet frames that are not based on characteristic
functions in the Fourier domain.  Recall that a wave packet system is
a the collection of functions that arises by letting a class of
translation, modulation, and scaling operators act on a fixed function
$\psi\in \ltr$. The precise setup is as follows. Given $a\in \R$, we
define the modulation operator $(E_af)(x)=\myexp{2\pi iax}f(x)$, and
(for $a> 0$) the scaling operator $(D_af)(x)=
a^{-1/2} f(x/a)$; these operators are unitary on
$\ltr$. Let $b>0$ and $\{(a_j, d_j)\}_{j\in J}$ be a countable set in
``scale/frequency'' space $\R^+\times \, \R$. The
\emph{wave packet system} generated by a function $\psi\in \ltr$ is
the collection of functions $\wpj$.

The key feature of wave packet system is that it allows
 us to \emph{combine} the Gabor structure and the wavelet
structure into one system that yields a very flexible analysis of
signals. For the particular parameter choice $(a_j,d_j):=(a^j,1)$ for $j \in
J=\Z$ and some $a \neq 0$, the wave packet system $\wpj$ simply
becomes the wavelet system $\{D_{a^j}T_{kb}\psi\}_{j,k\bbz}$ generated
by the function $\psi\in \ltr$. On the other hand, for the choice
$(a_j,d_j):=(1,aj)$ for $j \in J=\Z$ and some $a> 0$, we recover the
system $\set{T_{bk} E_{aj} \psi}_{j,k\in \Z}$ which is unitarily
equivalent with the Gabor system $\set{E_{aj} T_{bk} \psi}_{j,k\in
  \Z}$. Hence, we can consider both wavelet and Gabor systems as
special cases of wave packet systems. Furthermore, other choices of the
parameters $\{(a_j, d_j)\}_{j\in J}$, which intuitively control
how the scale/frequency information of a signal is analyzed, combine
Gabor and wavelet structure. Finally, the
translations by $b\Z$ allow for time localization of the wave packet
atom.


The generality of GSI systems is known to lead to some technical
issues. Indeed, local integrability conditions play an important role
in the theory of GSI systems as a mean to control the interplay
between the translation lattices $c_j \Z$ and the generators $g_j$, $j
\in J$. Our third main contribution is new insight into the role of
local integrability conditions. In particular, we will see that local
integrability conditions also play an important role for wave packet
systems, and that it is important to distinguish between the so-called
local integrability condition (LIC) and the weaker $\alpha$-LIC. This
is in sharp contrast to the case of Gabor and wavelet systems in
$L^2(\R)$, where one largely can ignore local integrability
conditions.

The paper is organized as follows. In Section 2, we introduce the
theory of GSI systems and extend the well-known duality conditions to certain
subspaces of $\ltr.$ In Section \ref{sec:lower-bound-calderon} we discuss various
technical conditions under which the
Calder\'on sum for a GSI frame is  bounded below by the lower
frame bound; applications to wavelet systems and
 Gabor systems are considered in Section \ref{sec:special-cases-gsi}. In
Section \ref{sec:constr-dual-gsi-frames} we provide explicit constructions of dual GSI frames
for certain subspaces of $\ltr.$ The general version of the result is technical, but we are
nevertheless able to provide concrete realizations of the results. Finally,
Section \ref{sec:wave-packet-systems} applies the key results of the paper to wave packet systems.
In particular we show that a successful analysis of such systems must be based on the $\alpha$-LIC
rather than the LIC. Furthermore, we provide explicit constructions of dual pairs of
wave packet frames.

We end this introduction by putting our
work in a perspective with other known results. C\'ordoba and
Fefferman~\cite{MR507783} considered continuous wave packet transforms in
$L^2(\R^n)$ generated by the gaussian which is well localized in time
and frequency. The results in \cite{MR507783} yield \emph{approximate} reproducing
formulas. In \cite{MR2066831, MR2038268} the authors constructs
frequency localized wave packet systems associated with  exact
reproducing formulas in terms of Parseval frames. However, these generators
are poorly localized in time as the generators are characteristic
functions in the frequency space. In this work we construct wave
packet dual frames well localized in time and frequency.

For an introduction to frame theory we refer to the books
\cite{MR1946982,MR1843717,Daubechies:1992:TLW:130655}.

%


\section{Preliminary results on GSI systems}
\label{sec:gener-shift-invar}

To set the stage, we will recall and extend some of the most important
results on GSI systems. Let $J$ be a finite or a countable index set. As already mentioned in the introduction, analysis of GSI
systems  $\gsi$ cover several of the cases considered in the literature. In case $c_j=c >0$ for all $j \in J$, the system
$\{T_{c k}g_j\}_{k \in \Z, j \in J}$ is a shift invariant (SI) system;
if one further takes $g_j=E_{aj}g, j \in J:=\Z$ for some $a>0$ and $g \in L^2(\R)$,
we recover the Gabor case. The
wavelet system $\{D_{a^j}T_{bk}\psi\}_{j,k\bbz}$ with $a> 0$ and
$b>0$ can naturally be
represented as a GSI system via
\begin{equation}
\label{eq:0803c}
\{D_{a^j}T_{bk}\psi\}_{j,k\bbz}=\gsi \quad  \text{ with }
c_j=a^jb, \, g_j=D_{a^j}\psi, \text{ for } j \in J=\Z.
\end{equation}
Note that this representation is non-unique, hence unless it is clear from the context, we will
always specify the choice of $c_j$ and $g_j$, $j \in J$.

The upper bound of the Calder\'on sum for GSI systems obtained by Hern\'andez, Labate,
and Weiss~\cite{MR1916862} only relies on the Bessel property. The precise
statement is as follows.
\begin{thm}[\!\!\cite{MR1916862}]
\label{thm:HLW-Bessel-bound}
Suppose the GSI system $\gsi$ is a Bessel sequence with bound $B$. Then
  \begin{equation} \label{eq:5301a} \sum_{j\in
      J}\frac{1}{c_j}|\hat{g}_j(\gamma)|^2 \le B \qquad \text{for a.e. }
    \gamma\in \R.
  \end{equation}
\end{thm}
Here, for $f \in L^1(\R)$, the Fourier transform is defined as
\[
\hat f(\gamma) = \int_{\R} f(x)\myexp{-2 \pi i
  \gamma x} \, \mathrm{d}x
  \]
with the usual extension to $L^2(\R)$.

In the special cases where $\gsi$ is a Gabor frame or a wavelet frame
with lower frame bound $A$,
it is known that $A$ is also a lower bound on the sum in~\eqref{eq:5301a}. For instance,
for a wavelet frame $\{D_{a^j}T_{bk}\psi\}_{j,k\bbz}= \{T_{a^jbk}D_{a^j}\psi\}_{j,k\bbz}$
with bounds $A$ and $B$, Chui and Shi~\cite{MR1199539} proved that
\begin{equation}
\label{eq:wavelet-calderon-bounds}
A\leq \sum_{j\in J}\frac{1}{b}|\hat{\psi}(a^j\gamma)|^2\leq B \qquad
\text{for a.e. }
    \gamma\in \R.
\end{equation}

\subsection{Frame theory for GSI systems}
\label{sec:frame-theory-gsi}

We will consider GSI frames $\gsi$ for certain closed subspaces of
$L^2(\R)$. To this end, for a measurable subset $S$ of $\R$ we define
\[
 \check{L}^2(S):=\setprop{f \in L^2(\R)}{\supp\hat{f} \subset S}.
\]
The set $S$ is usually some collection of frequency bands that are of
interest. In case one is not interested in subspaces of $L^2(\R)$,
simply set $S=\R$. If $S$ is chosen to be a finite, symmetric
interval around the origin, we obtain the important special case of
Paley-Wiener spaces. We will always assume that the generators $g_j$ satisfy that
$\supp\hat{g}_j \subset S$ for every $j \in J$. Note that this
guarantees that the GSI system $\gsi$ belongs to $\check{L}^2(S)$.

In order to check that a GSI system is a frame for $\check{L}^2(S)$
it is enough to check the
frame condition 
on a dense set in
$\check{L}^2(S)$. Depending on the given GSI system, we will fix a
measurable set $E\subset S$ whose closure has measure zero and
define the subspace $\cD_E$ by
\[
\cD_E:=\left\{f\in L^2(\R): \supp\hat{f} \subset S\setminus E \text{ is compact and }
  \hat{f} \in L^\infty(\R)\right\}.
\]
For example, for a Gabor system $\mts$ we can take $E$ to be the empty set, and for a wavelet system
 $\{D_{a^j}T_{bk}\psi\}_{j,k\bbz}$ we take $E=\{0\}$.

In order to consider frame properties for GSI systems we will need a local integrability condition,
introduced in \cite{MR1916862} and generalized in \cite{JakobsenReproducing2014}.

\begin{defn}
  Consider a GSI system $\gsi$ and let $E \in \mathcal{E}$, where
  $\mathcal{E}$ denotes the set of measurable subsets of $S \subset \R$ whose closure has
  measure zero.
  \begin{enumerate}[(i)]
  \item If
    \begin{equation}\label{eq:LIC}
    L(f):= \sum_{j\in J}\sum_{m\bbz}\frac{1}{c_j}\int_{\supp\hat{f}}|
      \hat{f}(\gamma+c_j^{-1}m) \hat{g}_j(\gamma)|^2\,\mathrm{d}\gamma<\infty
    \end{equation}
    for all $f\in {\cal D}_E$, we say that $\gsi$ satisfies the local integrability condition (LIC) with respect to the set $E$.
  \item $\gsi$ and $\gsi[h]$ satisfy the
     \emph{dual $\alpha$-LIC} with respect to $E$ if
    \begin{equation}\label{eq:dual-alpha-LIC}
    L'(f):=  \sum_{j\in J}\sum_{m\bbz}\frac{1}{c_j}\int_{\Bbb R}| \hat{f}(\gamma)\overline{\hat{f}(\gamma+c_j^{-1}m)
        \hat{g}_j(\gamma)} \hat{h}_j(\gamma+c_j^{-1}m)|\,\mathrm{d}\gamma<\infty
    \end{equation}
    for all $f\in{\cal D}_E$.  We say that $\gsi$ satisfies the
    $\alpha$-LIC with respect to $E$, if~\eqref{eq:dual-alpha-LIC}
    holds with $g_j=h_j$, $j \in J$.
  \end{enumerate}
\end{defn}
By an application of the Cauchy-Schwarz inequality, we see that if the
local integrability condition holds, then the $\alpha$-local
integrability condition also holds. 
Clearly, if a  local
integrability condition holds with respect to $E = \emptyset$, it
holds  with respect to any $E \in \cE$. 

In \cite{MR1916862} it is shown that any Gabor system satisfies the LIC for
$E=\emptyset$. To arrive at the same conclusion for SI
systems,  it suffices to assume that the system is a Bessel
sequence, see \cite{JakobsenReproducing2014}. In fact, it is not difficult to show the
following more general result.
\begin{lem}
 Consider the SI system $\csi$, and let $E \in \cE$. Then $\csi$
 satisfies the LIC with respect to $E$ if and only if the Calder\'on sum
 for $\csi$ is locally integrable on $\R\setminus E$, i.e.,
 \begin{equation}
\sum_{j\in J} \frac{1}{c}\abs{\hat{g}_j(\cdot)}^2 \in
\Lloc(\R\setminus E). \label{eq:SI-system-L1-loc}
\end{equation}
\end{lem}
Of course, one can leave out the factor $\frac{1}{c}$ in the
Calder\'on sum in \eqref{eq:SI-system-L1-loc}. Note that if $\csi$ is
a Bessel sequence, then, by Theorem~\ref{thm:HLW-Bessel-bound}, the
Calder\'on sum satisfies \eqref{eq:SI-system-L1-loc} for any $E \in \cE$.
Similarly, it was shown in \cite{MR2746669} that
a wavelet system $\{D_{a^j}T_{kb}\psi\}_{j,k\bbz}$ satisfies the LIC
with respect to $E=\set{0}$ if and only if
\begin{equation*}
  \sum_{j\bbz} | \widehat{\psi}(a^j \cdot )|^2 \in \Lloc(\R\setminus\{0\}).
\end{equation*}

Hern\'andez, Labate and Weiss~\cite{MR1916862} characterized duality for two
GSI systems satisfying the LIC. In \cite{JakobsenReproducing2014} Jakobsen and Lemvig
generalized this to not necessarily discrete GSI systems defined on a locally
compact abelian group and satisfying
the weaker $\alpha$-LIC. The following  generalization  to discrete GSI system in
$\check{L}^2(S)$ follows the original proofs closely, so we only sketch the proof.

\begin{thm} 
  \label{thm:HLW}
  Let $S \subset \R$. Suppose that $\gsi$ and $\gsi[h]$ are Bessel
  sequences in $\check{L}^2(S)$ satisfying the dual $\alpha$-LIC for
  some $E\in \cE$. Then $\gsi$ and $\gsi[h]$ are dual frames for $\check{L}^2(S)$ if and only if
  \begin{equation}
  \label{eq:char-eqns-dual}
 \sum_{j\in J : \alpha \in c_j^{-1} \Z} \frac1{c_j}\,
 \overline{\hat{g}_j(\gamma)} \, \hat{h}_j(\gamma+ \alpha)=
    \delta_{\alpha, 0}\, \chi_{S}(\gamma) \qquad \text{a.e.} \, \gamma\in \R
  \end{equation}
  for all $\alpha \in \bigcup\limits_{j \in J} c_j^{-1}\Z$.
\end{thm}

\begin{proof}
  For simplicity assume that $E=\emptyset$; the case of general $E$
  only requires few modifications  of the following proof. For $f \in
  \cD_E$ define the function $w_f:\R \to \C$ by
  \begin{equation}
    \label{eq:w-def-GSI}
   w_f(x)= \sum_{j \in J} \sum_{k \in \Z}
   \innerprod{\tran[x]f}{\tran[c_jk]g_j}  \innerprod{\tran[c_jk]h_j}{\tran[x]f}.
 \end{equation}
By \cite[Proposition 9.4]{MR1916862}  (the given proof also hold with the $\alpha$-LIC replacing the
LIC) we know that $w_f$
is a continuous, almost periodic function that coincides pointwise
with its absolutely
convergent Fourier series
\begin{equation}
  \label{eq:w-FS-GSI}
    w_f(x)= \sum_{\alpha \in \bigcup\limits_{j\in J} c_j^{-1}\Z} d_\alpha \,
   \mathrm{e}^{2 \pi i \alpha x},
\end{equation}
where
\begin{equation}
  \label{eq:hat-w-GSI}
  d_\alpha = \int_{\R} \hat{f}(\gamma) \overline{\hat{f}(\gamma+\alpha)} \, t_\alpha(\gamma) d\gamma.
\end{equation}
and $t_\alpha(\gamma)$ denotes the left hand side of
\eqref{eq:char-eqns-dual}.

Assume that \eqref{eq:char-eqns-dual} holds. Inserting
$t_\alpha(\gamma)=\delta_{\alpha, 0}\, \chi_{S}(\gamma)$ into
\eqref{eq:w-FS-GSI} for $x=0$ yields
\[
 \sum_{j \in J} \sum_{k \in \Z}
   \innerprod{f}{\tran[c_jk]g_j}  \innerprod{\tran[c_jk]h_j}{f}=w_f(0) =
   \int_S \absbig{\hat{f}(\gamma)}^2 d\gamma = \norm{f}^2.
\]
By a standard density argument, this completes the proof of the
``if''-direction.

Assume now that $\gsi$ and $\gsi[h]$ are dual frames
for $\check{L}^2(S)$. Then $w_f(x)=\norm{f}^2$ for all $f\in \cD_E$ and all $x\in \mr.$ By uniqueness of
Fourier coefficients of almost periodic functions, this only happens
if, for $\alpha \in \bigcup\limits_{j\in J} c_j^{-1}\Z$,
\begin{equation}
 d_0 = \norm{f}^2 \qquad \text{and} \qquad  d_\alpha  = 0 \quad
 \text{for  $\alpha \neq 0$}.\label{eq:42}
\end{equation}
Since $\cD_E$ is dense in $\check{L}^2(S)$, it follows from $d_0 =
\norm{f}^2 $ that $t_0(\gamma)=1$ for a.e. $\gamma \in S$.

Assume that $\alpha = c_j^{-1}k$ for some $j \in J$ and $k \in
\Z\setminus \{0\}$. For each $\ell \in \Z$, take
\begin{equation*}
  \hat{f}(\gamma) =
  \begin{cases}
     1          & \text{for } \gamma \in
       \itvcc{c_j^{-1}\ell}{c_j^{-1}(\ell+1)} \cap S, \\
     t_\alpha(\gamma) & \text{for } \gamma \in
       \itvcc{c_j^{-1}\ell-\alpha}{c_j^{-1}(\ell+1)-\alpha} \cap S, \\
     0          & \text{otherwise.}
   \end{cases}
\end{equation*}
Then $f \in \cD_E$ and
\begin{equation*}
  0=\int_{\R} \hat{f}(\gamma) \overline{\hat{f}(\gamma+\alpha)} \, t_\alpha(\gamma)
  \,d\gamma = \int_{\itvcc{c_j^{-1}\ell}{c_j^{-1}(\ell+1)} \cap S} \abs{t_\alpha(\gamma)}^2
  d\gamma.
\end{equation*}
Since $\ell \in \Z$ was arbitrarily chosen, we deduce
that $t_\alpha(\gamma)$ vanishes almost everywhere for $\gamma \in S$.

For a.e.\ $\gamma \notin S$ the assumption
$\supp \hat{g}_j \subset S$ implies that $t_\alpha(\gamma)=0$ for any $\alpha$. Summarizing,
we have shown that $t_\alpha(\gamma)=\delta_{\alpha, 0}\,
\chi_{S}(\gamma)$ for a.e. $\gamma \in \R$.
\end{proof}
In the characterization of tight frames, we can leave out the Bessel condition.
\begin{thm} 
  \label{thm:HLW-tight}
  Let $S \subset \R$ and $A>0$. Suppose that the GSI system $\gsi$
  satisfies the $\alpha$-LIC condition for some $E \in \cE$. Then $\gsi$ is a tight frame for $\check{L}^2(S)$ with frame bound $A$ if and only if
  \begin{equation*}
\sum_{j\in J : \alpha \in c_j^{-1} \Z} \frac1{c_j}\,
\overline{\hat{g}_j(\gamma)} \, \hat{g}_j(\gamma+ \alpha)=
    A\,\delta_{\alpha, 0}\, \chi_{S}(\gamma) \qquad \text{a.e.} \, \gamma
    \in \R
  \end{equation*}
  for all $\alpha \in \bigcup\limits_{j \in J} c_j^{-1}\Z$.
\end{thm}

For tight frames, Theorem~\ref{thm:HLW-tight} gives information about
the Calder\'on sum. Indeed, it follows immediately from Theorem~\ref{thm:HLW-tight} that if $\gsi$ is a tight frame with
constant $A$ satisfying the $\alpha$-local integrability condition, then
 \[
\sum_{j \in J} \frac{1}{c_j}\vert \hat{g}_j(\gamma)\vert^2 = A \qquad \text{ a.e. } \gamma \in S.
\]


Finally, the following result allows us to construct frames without worrying about
technical local integrability conditions. In fact, the
condition~\eqref{eq:CC-condition} below implies that the
$\alpha$-LIC with respect to any set $E\in \cE$ is satisfied.

\begin{thm}
  \label{thm:suff-cond-for-A-and-B-bound} Consider the generalized
  shift invariant system \gsi.
\begin{enumerate}[(i)]
\item If
\begin{equation} \label{eq:CC-condition} B:= \esssup_{\gamma\in S}
  \sum_{j\in J} \sum_{m \in \Z}  \frac{1}{c_j} \absbig{ \hat{g}_j(\gamma) \hat{g}_j(\gamma+c_j^{-1}m)} < \infty,
\end{equation}
then \gsi\ is a Bessel sequence in $\check{L}^2(S)$ with bound
$B$.
\item Furthermore, if also
\[ A:= \essinf_{\gamma\in S} \Big( \sum_{j\in J} \frac{1}{c_j}  \abs{\hat
g_j(\gamma)}^2 - \sum_{j\in J} \sum_{0 \neq m \in \Z}
\frac{1}{c_j}  \abs{\hat{g}_j(\gamma) \hat{g}_j(\gamma+c_j^{-1}m) } \Big) > 0,
\]
then \gsi\ is a frame for $\check{L}^2(S)$ with bounds $A$ and $B$.
\end{enumerate}
\end{thm}
The proof of Theorem~\ref{thm:suff-cond-for-A-and-B-bound} is a straightforward
modification of the standard proof for $\ltr$ (see, e.g., \cite{MR2420867,MR3313475,JakobsenReproducing2014}).

\section[A lower bound of Calder\'on sum for GSI systems]{A lower bound of the Calder\'on sum for GSI systems}
\label{sec:lower-bound-calderon}

Following a construction by Bownik and Rzeszotnik~\cite{MR2066821} and
Kutyniok and Labate~\cite{MR2283810}, we first show that the
Calder\'on sum for arbitrary GSI frames is not necessarily bounded
from below by the lower frame bound.
\begin{ex} \label{exa:gsi-bownic}
Consider the orthonormal basis $\{E_kT_m\chi_{[0,1[}\}_{k,m\bbz}$, an integer $N\geq 3$, and the
lattices  $\Gamma_j=N^j\Z, \, j\in \Bbb N$. There exists a sequence $\{t_i\}_{i=1}^\infty$ such that
\[
\bigcup_{i=1}^\infty(t_i+\Gamma_i)=\Z,\quad (t_i+\Gamma_i)\cap (t_j+\Gamma_j)=\emptyset\quad
\text{for }i\neq j,
\] i.e., $\Z$ can be decomposed  into translates of
the sparser lattices $\Gamma_j.$ It follows that
\[
\{{\cal F}^{-1}E_kT_m\chi_{[0,1[}\}_{k,m\bbz}=\{T_k E_m{\cal F}^{-1}\chi_{[0,1[}\}_{k,m\bbz}=\{T_{N^jk} T_{t_j}E_m{\cal F}^{-1}\chi_{[0,1[}\}_{k,m\bbz, j\in \mn}.
\]
Hence, the GSI system defined by
\[
c_{(j,m)}=N^j \; \text{ and } \;  g_{(j,m)} =  T_{t_j}E_m{\cal F}^{-1}\chi_{[0,1[}, \quad \text{for } (j,m)
\in J=\N \times \Z,
 \]
is an orthonormal basis and therefore, in particular, a Parseval frame. Since
 \begin{align*}
 \sum_{j=1}^\infty \sum_{m \bbz} \frac1{c_{(j,m)}} \absbig{\hat{g}_{(j,m)}(\gamma)}^2 &=
   \sum_{j=1}^\infty \sum_{m \bbz} \frac1{N^j} |{\cal F}T_{t_j}E_m{\cal F}^{-1}\chi_{[0,1[}(\gamma)|^2 \\
   &= \sum_{j=1}^\infty\sum_{m\bbz}\frac{1}{N^j}|E_{t_j} \chi_{\itvco{m}{m+1}}(\gamma)|^2
   = \sum_{j=1}^\infty\frac{1}{N^j}|\chi_{\Bbb R(\gamma)}|^2=\frac{1}{N-1},
 \end{align*}
 we conclude that the Calder\'on sum~\eqref{eq:5301a} is not bounded below by the lower frame bound
 $A=1$ whenever $N\geq3$. On the other hand, we see that the Calder\'on sum is indeed bounded from
 above by the (upper) frame bound $1$ as guaranteed by
 Theorem~\ref{thm:HLW-Bessel-bound}.
\end{ex}

We will now provide a technical condition on a frame $\gsi$ that implies that the Calder\'on sum
in~\eqref{eq:5301a} is bounded by the lower frame bound. The proof generalizes the argument in
\cite{MR1199539}.
\begin{thm}\label{thm:loc-lowerbound}
  Assume that
  \begin{equation}
    \label{eq:loc}
    \sum_{j\in J}|\hat{g}_j(\cdot)|^2\in \Lloc(\R\setminus E).
  \end{equation}
  If the GSI system $\gsi$ is a frame for $L^2(\R)$ with lower bound $A$,
  then
  \begin{equation}\label{eq:lowerbound}
    A\leq \sum_{j\in J}\frac{1}{c_j}|\hat{g}_j(\gamma)|^2  \qquad \text{a.e.} \, \gamma\in \R.
  \end{equation}
\end{thm}
\begin{proof}
  The assumption that the function $\sum_{j\in J}|\hat{g}_j(\cdot)|^2 $ is locally
  integrable in $\R\setminus E$ implies, by  the Lebesgue
  differentiation theorem, that the set of its Lebesgue point is dense in $\R\setminus
  E$.  Let $\omega_0\in\Bbb R\setminus \overline{E}$ be a Lebesgue
  point. Choose $\eps'>0$ such that $[\omega_0-\eps',\omega_0+\eps']\subseteq \R\setminus
  \overline{E}$.  By assumption, we have
  \[
  \int_{\omega_0-\eps'}^{\omega_0+\eps'} \sum_{j\in J}|\hat{g}_j(\gamma)|^2\,\mathrm{d}\gamma<\infty.
  \]
  This means that for every $\eta>0$, there exists a  finite set $J'\subset J$ such that
  \begin{equation}\label{eq:eta}
    \sum_{ j\in J\setminus J'}\int_{\omega_0-\eps'}^{\omega_0+\eps'}|\hat{g}_j(\gamma)|^2\,\mathrm{d}\gamma<\eta.
  \end{equation}
  Now define $M:=\max_{j\in J'}c_j$. Let $0<\eps<\min\{\eps', \frac{M^{-1}}{2}\}$. It is clear
  that~\eqref{eq:eta} also holds for every $\eps>0$ with $\eps<\eps'$.

Using Lemma 20.2.3 of \cite{MR1946982}, we have, for $f\in{\cal D}_E$,
  \begin{equation}\label{eq:main}
    A\|f\|^2\leq \sum_{j\in J}\sumk\frac{1}{c_j}\int_{\Bbb R}\hat{f}(\gamma)
    \overline{\hat{f}(\gamma+c_j^{-1}k)\hat{g}_j(\gamma)} \hat{g}_j(\gamma+c_j^{-1}k)\,\mathrm{d}\gamma.
  \end{equation}
  Consider $\hat{f}=\frac{1}{\sqrt{2\eps}}\chi_{K}$, where $K=[\omega_0-\eps,\omega_0+\eps]$. Since
  $f\in{\cal D}_E$, by inequality~\eqref{eq:main}, we have
  \begin{align*}
    A&\leq
    \sum_{j\in J}\sumk\frac{1}{2\eps c_j}\int_{K\cap (K-c_j^{-1}k)}
    \overline{\hat{g}_j(\gamma)}\hat{g}_j(\gamma+c_j^{-1}k)\,\mathrm{d}\gamma\\
    &= \sum_{j\in J'}\sumk\frac{1}{2\eps c_j}\int_{K\cap
      (K-c_j^{-1}k)}\overline{\hat{g}_j(\gamma)}\hat{g}_j(\gamma+c_j^{-1}k)d\gamma\\ & \qquad +
    \sum_{j\in J\setminus J'}\sumk\frac{1}{2\eps c_j}\int_{K\cap (K-c_j^{-1}k)}
    \overline{\hat{g}_j(\gamma)} \hat{g}_j(\gamma+c_j^{-1}k)\,\mathrm{d}\gamma\\
    &=:S_1+S_2.
  \end{align*}
  For $j\in J'$, we have $2\eps< M^{-1}\le c_j^{-1}|k|$ for all $k\bbz\setminus\{0\}$. Hence
  \[
  K\cap (K-c_j^{-1}k)=\emptyset,\quad k\bbz\setminus\{0\},j\in J'.
  \]
  Therefore
  \begin{align}\label{eq:0.6}
    S_1 =\sum_{j\in J'}\frac{1}{2\eps c_j}\int_{K}|\hat{g}_j(\gamma)|^2\,\mathrm{d}\gamma.
  \end{align}
  In particular, $S_1$ is a non-negative number.  For $j\in J\setminus J'$, $K\cap
  (K-c_j^{-1}k)=\emptyset$ if $|k|> 2\eps c_j$. Therefore,
  \[
  S_2 \leq \sum_{j\in J\setminus J'}\sum_{|k|\leq 2\eps c_j}\frac{1}{2\eps c_j}\int_{K\cap
    (K-c_j^{-1}k)}|\overline{\hat{g}_j(\gamma)} \hat{g}_j(\gamma+c_j^{-1}k)|\,\mathrm{d}\gamma.
  \]
  Now by the Cauchy-Schwarz inequality, we have
  \begin{align}\label{eq:3.11}
    S_2&\leq \sum_{j\in J\setminus J'}\sum_{|k|\leq 2\eps c_j}\frac{1}{2\eps c_j} \Bigl(\int_{K\cap
      (K-c_j^{-1}k)}|\hat{g}_j(\gamma)|^2\,\mathrm{d}\gamma \Bigr)^{\frac{1}{2}}
    \Bigl(\int_{(K+c_j^{-1}k)\cap K}|\hat{g}_j(\gamma)|^2 \,\mathrm{d}\gamma\Bigr)^\frac{1}{2}\nonumber\\
    &\leq \sum_{j\in  J\setminus J'}(4\eps c_j+1)\frac{1}{2\eps c_j}\int_{K}| \hat{g}_j(\gamma)|^2\,\mathrm{d}\gamma\nonumber\\
    &= 2\sum_{j\in J\setminus J'}\int_{K}|\hat{g}_j(\gamma)|^2\,\mathrm{d}\gamma+ \sum_{j\in J\setminus
      J'}\frac{1}{2\eps c_j}\int_{K}| \hat{g}_j(\gamma)|^2\,\mathrm{d}\gamma\nonumber\\
    &\leq 2\eta + \sum_{j\in J\setminus J'}\frac{1}{2\eps
      c_j}\int_{K}|\hat{g}_j(\gamma)|^2\,\mathrm{d}\gamma.
        \end{align}
  Since $A\leq S_1+S_2$, it follows from~\eqref{eq:0.6} and~\eqref{eq:3.11} that
  \begin{align*}
    A&\leq \sum_{j\in J'}\frac{1}{2\eps c_j}\int_{K}|\hat{g}_j(\gamma)|^2\,\mathrm{d}\gamma+2\eta +
    \sum_{j\in  J\setminus J'}\frac{1}{2\eps c_j} \int_{K}|\hat{g}_j(\gamma)|^2\,\mathrm{d}\gamma\\
    &= \sum_{j\in J}\frac{1}{2\eps c_j}\int_{K}|\hat{g}_j(\gamma)|^2\,\mathrm{d}\gamma + 2\eta\\
    &= \frac{1}{2\eps}\int_{\omega_0-\eps}^{\omega_0+\eps}\sum_{j\in
      J}\frac{1}{c_j}|\hat{g}_j(\gamma)|^2\,\mathrm{d}\gamma + 2\eta.
        \end{align*}
  Letting $\eps\rightarrow 0$, we arrive at
  \[
  A\leq \sum_{j\in J}\frac{1}{c_j}|\hat{g}_j(\omega_0)|^2+2\eta.
  \]
  Since $\eta>0$ is arbitrary, the proof is complete.
\end{proof}
In order to check the condition~\eqref{eq:loc}, it is enough to consider the partial sum over the
$j\in J$ for which $c_j>M$ for some $M>0$:
\begin{prop}
  \label{prp:sum-over-cj-M}
  Suppose that $\gsi$ is a Bessel sequence with bound $B$. Then~\eqref{eq:loc} holds if and only if
  there exist some $M>0$ such that
  \begin{equation}
    \label{eq:loc-with-M}
    \sum_{\{j:c_j>M\}}\absbig{\hat{g}_j(\cdot)}^2\in \Lloc(\R \setminus E).
  \end{equation}
\end{prop}
\begin{proof}
  To see this, consider any $M>0$. Then
  \begin{align*}
    \sum_{\{j:c_j>M\} }|\hat{g}_j(\gamma)|^2 \le    \sum_{j\in J}|\hat{g}_j(\gamma)|^2&=  \sum_{\{j:c_j\leq M\} }|\hat{g}_j(\gamma)|^2+\sum_{\{j:c_j>M\} }|\hat{g}_j(\gamma)|^2\\
    &\leq M\sum_{\{j:c_j\leq M\} } \frac{1}{c_j}|\hat{g}_j(\gamma)|^2+\sum_{\{j:c_j>M\} }|\hat{g}_j(\gamma)|^2\\
    & \leq MB+\sum_{\{j:c_j>M\} }|\hat{g}_j(\gamma)|^2.
  \end{align*}
\end{proof}


Let us take a closer look at the essential condition~\eqref{eq:loc}.  First we note that
in Example~\ref{exa:gsi-bownic} this condition is indeed violated: in fact, a slight modification of the
calculation in Example~\ref{exa:gsi-bownic} shows that the infinite series in~\eqref{eq:loc} is divergent for
all $\gamma \in\R$. The next result shows that under mild regularity conditions on $c_j$ and $g_j$,
condition~\eqref{eq:loc} is equivalent to the LIC. However, in practice, condition~\eqref{eq:loc},
or rather condition~\eqref{eq:loc-with-M}, is often much easier to work with than the LIC.

\begin{prop}
  \label{LICM}
  Let $E \in \cE$. Suppose that the Calder\'on sum for $\gsi$ is locally integrable on $\R\setminus E$, i.e.,
  \[
  \sum_{j \in J} \frac{1}{c_j}\vert \hat{g}_j(\cdot)\vert^2 \in \Lloc(\R\setminus E),
  \]
  Then the GSI system $\gsi$ satisfies the LIC with respect to the set $E$ if and only
  if~\eqref{eq:loc} holds.
\end{prop}
\begin{proof}
  Assume that
  \begin{equation}
    \label{eq:licc}
    L(f)=\sum_{j\in J}\sum_{m\bbz}\frac{1}{c_j}\int_{\supp\hat{f}}| \hat{f}(\gamma+c_j^{-1}m) \hat{g}_j(\gamma)|^2\,\mathrm{d}\gamma<\infty
  \end{equation}
  for every $f\in{\cal D}_E$. Let $K$ be a compact set in $\R\setminus E$, i.e., $K \subset
  [c,d]\setminus E$, and let $\hat{f}=\chi_K$ in~\eqref{eq:licc}. Then for each $j\in J$, the set
  $K\cap (K-c_j^{-1}m)$ can only be nonempty if $|m|\leq c_j(d-c)$. Hence for the inner sum
  in~\eqref{eq:licc}, we have
  \begin{align*}
    \sum_{m\bbz} \frac{1}{c_j}\int_{K\cap (K-c_j^{-1}m)}|\hat{g}_j(\gamma)|^2\,\mathrm{d}\gamma &=
    \sum_{|m|\leq c_j(d-c)} \frac{1}{c_j} \int_{K\cap (K-c_j^{-1}m)}|\hat{g}_j(\gamma)|^2\,\mathrm{d}\gamma\\
    &\geq \frac{[c_j(d-c)]}{c_j}\int_{K}|\hat{g}_j(\gamma)|^2\,\mathrm{d}\gamma\\
    &\geq (d-c)\int_{K}|\hat{g}_j(\gamma)|^2d
    \gamma-\frac{1}{c_j}\int_{K}|\hat{g}_j(\gamma)|^2 \,\mathrm{d}\gamma,
  \end{align*}
  where we for the first inequality use that the union of the sets $K
  \cap (K-c_j^{-1}m)$, $|m|\leq [c_j(d-c)]$, contains $[c_j(d-c)]$
  copies of $K$.
  Hence,
  \begin{align*}
    (d-c)\sum_{j\in J}\int_{K}|\hat{g}_j(\gamma)|^2\,\mathrm{d}\gamma &\leq L(f)+ \sum_{j\in
      J}\frac{1}{c_j}\int_{K}|\hat{g}_j(\gamma)|^2\,\mathrm{d}\gamma<\infty.
  \end{align*}
  Thus~\eqref{eq:loc} holds.  Conversely, assume that~\eqref{eq:loc} holds. For $f\in{\cal D}_E$,
  let $R>0$ such that $\supp\hat{f}\subseteq\{\ga:|\ga|\leq R\}$.  Therefore the inner sum in
  \eqref{eq:licc} has at most $4Rc_j+1$ terms. By splitting the sum into two terms, we have
  \begin{align*}
    L(f)&\leq \|\hat{f}\|^2_\infty \sum_{j\in J}\sum_{|m|\leq 2Rc_j}\frac1{c_j}\int_{\supp\hat{f}}|\hat{g}_j(\gamma)|^2\,\mathrm{d}\gamma\\
    &\leq 4R\|\hat{f}\|^2_\infty \sum_{j\in J}\int_{\supp\hat{f}}|\hat{g}_j(\gamma)|^2\,\mathrm{d}\gamma+ \|\hat{f}\|^2_\infty
    \sum_{j\in J}\frac1{c_j}\int_{\supp\hat{f}}|\hat{g}_j(\gamma)|^2\,\mathrm{d}\gamma<\infty.
  \end{align*}
\end{proof}
Using Proposition~\ref{LICM}, the following result is now just a reformulation of
Theorem~\ref{thm:loc-lowerbound}.
\begin{cor}\label{thm:lic-lower-bound}
  Assume that the LIC with respect to some set $E \in \cE$ holds for the GSI system $\gsi$.
If $\gsi$ is a frame for $L^2(\R)$ with lower bound $A$,
then \eqref{eq:lowerbound} holds.
\end{cor}

\begin{rem}
\label{rem:lower-bound-for-subpspace}
  Corollary~\ref{thm:lic-lower-bound} also holds for GSI frames for
  $\check{L}^2(S)$ in which case the Calder\'on sum is bounded from
  below by $A$ on $S$ and zero otherwise.
\end{rem}

The next example shows that the Calder\'on sum might be bounded from below by the lower frame bound,
even if neither the technical
condition~\eqref{eq:loc} nor the LIC is satisfied.
\begin{ex}
  \label{exa:loc-not-necessary}
  Consider the GSI system $\{ T_{N^jk}g_{j,p,\ell}\}_{k,\ell \bbz,j\in\Bbb N,p\in{\cal P}_j}$,
  where
  \bes \mathcal{P}_j=\{1,\ldots,N^j-1\}, \, \, \hat{g}_{j,p,\ell}=(N-1)^{1/2} T_\ell
  \chi_{[\frac{p}{N^j},\frac{p+1}{N^j}]}, \, \, c_{j,p,\ell}=N^j.\ens This system satisfies the
  $\alpha$-LIC. To see this, fix $f\in \cD_{\emptyset}$. We then have
  \begin{align*}
    L'(f)&= \sum_{j\in\Bbb N} \sum_{p\in{\cP}_j} \sum_{\ell \in\Z}
    \sum_{m\bbz}\frac{1}{N^j}\int_{\Bbb
      R}|\hat{f}(\gamma)\overline{\hat{f}(\gamma+\frac{m}{N^j})\hat{g}_{j,p,\ell}(\gamma
      )}\hat{g}_{j,p,\ell} (\gamma -\frac{m}{N^j}) |\,\mathrm{d}\gamma\\
    &\leq(N-1)\|\hat{f}\|_\infty^2 \sum_{j\in\Bbb N} \sum_{p\in{\cP}_j} \sum_{\ell \in\Z}
    \sum_{m\bbz}\frac{1}{N^j}\int_{\supp\hat{f}}  \chi_{[\frac{p}{N^j},\frac{p+1}{N^j}]}(\gamma -\ell)
    \chi_{[\frac{p}{N^j},\frac{p+1}{N^j}]} (\gamma -\ell-\frac{m}{N^j})\,\mathrm{d}\gamma\\
    &=(N-1)\|\hat{f}\|_\infty^2  \sum_{j\in\Bbb N} \sum_{p\in{\cP}_j} \sum_{\ell \in\Z}  \frac{1}{N^j}\int_{\supp\hat{f}} \chi_{[\frac{p}{N^j},\frac{p+1}{N^j}]}(\gamma -\ell)\,\mathrm{d}\gamma\\
    &=(N-1)\|\hat{f}\|_\infty^2 \sum_{j\in\Bbb N} \frac{1}{N^j}\int_{\supp\hat{f}} \chi_{\Bbb
      R}(\gamma)\,\mathrm{d}\gamma\\
    &=\|\hat{f}\|_\infty^2 \int_{\supp\hat{f}} \chi_{\Bbb R}(\gamma)\,\mathrm{d}\gamma<\infty.
  \end{align*}
  The considered GSI system is also a Parseval frame. To see this, using Proposition~\ref{thm:HLW-tight}, it is
  sufficient to show that
  \begin{equation}
    \label{eq:t-alpha}
    \sum_{j\in J_{\alpha}} \sum_{p\in{\cP}_j} \sum_{\ell \in\Z} \frac1{N^j}\,
    \overline{\hat{g}_{j,p,\ell}(\gamma)} \hat{g}_{j,p,\ell}(\gamma+ \alpha)= \delta_{\alpha, 0} \qquad \text{a.e.} \, \gamma\in \R.
      \end{equation}
  Let $\Lambda = \{ N^{-j} n \ : \ j\in J, n\bbz\};$ and, for $\alpha \in \Lambda$, let
  \bes J_\alpha=
  \{j\in J \ : \ \exists n\bbz \ \mbox{such that} \ \alpha= N^{-j} n\}.\ens Assume that
  $0\neq\alpha=\frac{n}{N^{j_0}}$ and $n\leq N^{j_0}$. Then $J_{\alpha}\subseteq\{j:j\geq
  j_0\}$. But for each $j \geq j_0$ and $\gamma \in \Bbb R$ we have
  \[
  \chi_{[\frac{p}{N^j},\frac{p+1}{N^j}]}(\gamma -\ell) \chi_{[\frac{p}{N^j},\frac{p+1}{N^j}]} (\gamma
  -\ell-\frac{n}{N^{j_0}})=0.
  \]
  Hence, \eqref{eq:t-alpha} is satisfied for $\alpha\neq 0$. Now, consider~\eqref{eq:t-alpha} with
  $\alpha=0$:
  \begin{align*}
    \sum_{j\in \Bbb N}\sum_{ p\in{\cal P}_j}\sum_{ \ell\in\Z} \frac1{N^j}\,
    |\hat{g}_{j,p,\ell}(\gamma)|^2
    &=(N-1)  \sum_{j\in \Bbb N}\sum_{ p\in{\cal P}_j}\sum_{ \ell\in\Z}   \frac1{N^j}    \,  \chi_{[\frac{p}{N^j},\frac{p+1}{N^j}]} (\gamma -\ell)\\
    &=(N-1)\sum_{j\in \Bbb N}\frac1{N^j}\,  \chi_{\Bbb R}(\gamma)=1.
  \end{align*}
  One can easily show that this Parseval frame does not satisfies condition~\eqref{eq:loc}, and
  consequently by Proposition~\ref{LICM}, it cannot satisfy the LIC
  for any $E \in \cE$, while the
  inequality~\eqref{eq:lowerbound} holds with $A=1$. 
\end{ex}

\section{Special cases of GSI systems}
\label{sec:special-cases-gsi}

We will now show that Theorem~\ref{thm:loc-lowerbound} indeed generalizes the known results for
wavelet and Gabor systems. First, for regular wavelet systems the condition~\eqref{eq:loc-with-M} is always satisfied for $E=\{0\}$:
\begin{lem}
\label{lem:wavelets-locM}
  Let $\psi\in L^2(\R)$ and $a>1, b>0$. Consider the wavelet
  system $\set{D_{a^j}T_{bk}\psi},$ written on the form \eqref{eq:0803c}. Then there exists $M\in\R$ such that
  \eqref{eq:loc-with-M} holds for $E=\{0\}$, i.e.,
  \begin{equation}
    \label{eq:103a}
    \sum_{\{j: c_j>M\}}|\hat{g_j}(\cdot)|^2\in \Lloc(\Bbb
    R\setminus\{0\}).
  \end{equation}
\end{lem}
\begin{proof}
  Assume that $K$ is a compact subset of $\R\setminus\{0\}$. Let
  $M=ab;$ then $a^jb>M$ if and only if $j>0$. Now,
  \[
  \sum_{j>0}\int_K |\hat{g_j}(\gamma)|^2 \,\mathrm{d}\gamma = \sum_{j>0}\int_K
  |D_{a^{-j}}\hat{\psi}(\gamma)|^2\,\mathrm{d}\gamma = \sum_{j>0}\int_K
  a^{j}|\hat{\psi}(a^j\gamma)|^2\,\mathrm{d}\gamma =\sum_{j>0}\int_{a^{j}K}
  |\hat{\psi}(\gamma)|^2\,\mathrm{d}\gamma.
  \]
  Since $K$ is a compact subset of $\R\setminus\{0\}$, one can find $L,R>0$ such that
  $K\subset\{\gamma:\frac{1}{R}<|\gamma|<R\}$ and $a^L>R^2$. Hence if $j-j_0\geq L$, then
  $a^{j_0}K\cap a^jK=\emptyset$.
Thus the family of subsets
  $\{a^jK\}_{j>0}$ can be considered as a finite union of mutually disjoint sets,
  \[
  \{a^jK\}_{j>0}=\bigcup_{i=1}^{L}
  \{a^{i+kL}K\}_{k=0}^\infty,
  \]
 Therefore,
  \[
  \sum_{j>0}\int_K |\hat{g_j}(\gamma)|^2 \,\mathrm{d}\gamma=\sum_{j>0}\int_{a^{j}K}
  |\hat{\psi}(\gamma)|^2\,\mathrm{d}\gamma\leq L\int_{\R} |\hat{\psi}(\gamma)|^2\,\mathrm{d}\gamma<\infty,
  \]
  which implies that~\eqref{eq:103a} holds.
\end{proof}
From Theorem~\ref{thm:loc-lowerbound} we can now recover the lower bound
in~\eqref{eq:wavelet-calderon-bounds} for wavelet frames. Assume that
$\set{D_{a^j}T_{bk}\psi}_{j,k\in \Z}$ is a frame for $L^2(\R)$ with bounds $A$ and $B$. By
Lemma~\ref{lem:wavelets-locM} and Proposition~\ref{prp:sum-over-cj-M}, the wavelet system satisfies
\eqref{eq:loc}; thus \eqref{eq:wavelet-calderon-bounds} holds by Theorem~\ref{thm:loc-lowerbound}.


To establish the lower bound of the Calder\'on sum for irregular wavelet
systems $\irw$, first obtained by
Yang and Zhou \cite{MR2038269}, we have to work a little harder. We mention that the result by Yang and
Zhou covers the more general setting of irregular dilations \emph{and} translations.
Recall that a sequence $\aj$ of positive numbers is said to be logarithmically separated by $\lambda>0$
if $\frac{a_{j+1}}{a_j}\geq\lambda$ for
  each $j\in\Z$.
\begin{lem}
\label{lem:log-seperated}
If a sequence $\aj$ of positive numbers is logarithmically separated by $\lambda>0$, then for each
$\psi\in \ltr$ and every compact subset $K\subset \R\setminus \{0\}$, we have
\[
\sumj \int_{a_jK}|\hp(\gamma)|^2 \, d \ga<\infty.
\]
\end{lem}
\begin{proof}
  Without loss of generality, assume that $\aj$ is an increasing sequence and
  $K\subset\R_+\setminus\{0\}$. There exist positive number $c,d$ such that $K\subset [c,d]$.
   Take $r\in\N$ such that $\lambda^{r-1}> \frac{d}{c}$; then
  \[
  \frac{a_{j+r}}{a_j}=\frac{a_{j+r}}{a_{j+r-1}} \frac{a_{j+r-1}}{a_{j+r-2}}\cdots\frac{a_{j+1}}{a_j}\geq
  \lambda^{r-1}> \frac{d}{c}.
    \]
  Hence $a_{j+r} c> a_jd$. This shows that $a_jK\cap a_{j+r}K=\emptyset$ for all $j\in\Z$. By a similar argument, $a_jK\cap a_{j-r}K=\emptyset$. Therefore
  \[
  \sumj \int_{a_jK}|\hp(\gamma)|^2 \, d \leq 2r \int_{\R}|\hp(\gamma)|^2 \, d \gamma<\infty.
  \]
\end{proof}
We will now show that any wavelet system (with regular translates) that form a Bessel sequence
automatically satisfies the LIC:
\begin{prop}
  Let $\psi\in\ltr$ and $\aj$ be a sequence in $\R_+$. If $\irw$ is a Bessel sequence, then it
  satisfies the LIC with respect to $E=\{0\}$.
\end{prop}
\begin{proof}
  Define $I_n=( 2^{n-\frac1{2}}, 2^{n+\frac1{2}}]$, for $n\in\Z$. By Lemma~1 in \cite{MR2038269}, there exists $M>0$ such that for each $n\in\Z$, the number of $j$ which $a_j$
  belongs to $I_n$ is less than $M$. On the other hand, each point of $I_{2n}$ is logarithmically
  separated with points from the interval $I_{2m}$ for $m,n\in\Z$ and $m\neq n$. Similarly a point
  of $I_{2n+1}$ is logarithmically separated with points from the interval $I_{2m+1}$ for $m,n\in\Z$
  and $m\neq n$. Hence we can consider $\aj$ as a disjoint union of
  finitely many  logarithmically separated subsets $\{a_j\}_{j\in J_i}, \ i=1,\cdots,N$.  Let $K\subset\R\setminus\{0\}$ be a
  compact set. By Lemma~\ref{lem:log-seperated}, we know that $\sum_{j\in J_i}
  \int_{a_jK}|\hp(\gamma)|^2<\infty$ for each  $i=1,\cdots,N$; it follows that
  \begin{align*}
    \sumj \int_{K}|a_j^{\frac{1}{2}}\hp(a_j\gamma)|^2 d\ga&=\sumj
    \int_{a_jK}|\hp(\gamma)|^2d\ga=\sum_{i=1}^N\sum_{j\in J_i} \int_{a_jK}|\hp(\gamma)|^2d\ga< \infty,
  \end{align*} as desired.
  \end{proof}


\begin{cor}
  Suppose that $\irw$ is a frame with lower bounds $A>0$. Then
  $\irw$ satisfies the LIC with respect to $E=\{0\}$ and
  \begin{equation*}
    A\leq \sumj \frac{1}{b}|\hp(a_j\gamma)|^2  \qquad \text{for a.e. } \gamma \in\R.
  \end{equation*}
\end{cor}

For other variants of GSI systems, covering the Gabor case, we also have the desired bounds of the
Calder\'on sum immediately from the frame property.
\begin{cor}
  Assume that $\gsi$ is a frame with lower bound $A>0$. If the sequence $\{c_j\}_{j\bbz}$ is
  bounded above, then
\[
A\leq  \sum_{j \in J} \frac{1}{c_j}\vert \hat{g}_j(\gamma)\vert^2  \qquad
\text{for a.e.} \, \gamma\in \R.\]
\end{cor}
\begin{proof}
  Assume that there exists $M>0$ such that $0<c_j\leq M$, for all $j\bbz$. Using the proof of
  Theorem~\ref{thm:loc-lowerbound}, we have $I_1=\Z$ and $I_2=\emptyset$. Hence by letting $\eps\rightarrow
  0$ in~\eqref{eq:0.6}, we have the result.
\end{proof}

\section{Constructing dual GSI frames}
\label{sec:constr-dual-gsi-frames}

We now turn to the question of how to obtain dual pairs of frames.  Indeed, we present a flexible construction procedure that
yields dual GSI frames for $\check{L}^2(S)$, where $S \subset \R$ is
any countable collection of frequency bands. The precise choice of
$S$ depends on the application; we refer to
\cite{MR2854065} for an implementation and applications of GSI systems within
audio signal processing. Our construction relies on a certain
partition of unity, closely related to the Calder\'on sum, and unifies
similar constructions of dual frames with Gabor, wavelet, and
Fourier-like structure in
\cite{MR2481501,Lemvig2012,MR2407863,MR2587577,MR3061703}. Due to its generality
the method will be technically involved; however, we will show that we nevertheless
are able to extract the interesting cases from the general setup.

Methods in wavelet theory and Gabor analysis, respectively, share many
common features. However, the decomposition into frequency bands is
very different for the two approaches.  To handle these differences in
one unified construction procedure, we need a very flexible setup. In
order to motivate the setup, we first consider the Gabor case and
the wavelet case more closely in the following example.

\begin{ex}
\label{exa:setup-1}
  For a GSI system $\gsi$ to be a frame for $\check{L}^2(S)$ it is necessary that
  the union of the sets $\supp \hat{g}_j$, $j \in J$, covers the  frequency domain $S$. This is an easy
  consequence of Remark~\ref{rem:lower-bound-for-subpspace} and the fact that the Fourier transform of $\gsi$ is
  $\setsmall{E_{c_j k}\hat{g}_j}_{k\in \Z,j\in J}$. For simplicity we
  here consider only $S=\R$.

  As we later want to apply the construction procedure for wave packet
  systems, it is necessary that the setup covers as well constructions of
  bandlimited dual wavelet as dual Gabor frames. For wavelet systems
  take the dyadic Shannon wavelet for $L^2(\R)$ as an example.  In
  this case, we split the frequency domain $S=\R$ in two sets
  $S_0=\itvoc{-\infty}{0}$ and $S_1=\itvoo{0}{\infty}$.
  The support of the dilates of the Shannon wavelet is $\supp
  \hat{g}_j=\supp D_{2^{-j}}\hat{\psi}=\itvcc{-2^{-j}}{-2^{-j-1}} \cup
  \itvcc{2^{-j-1}}{2^{-j}}$. To control the support of $\hat{g}_j$ we
  will define certain knots at the dyadic fractions $\set{\pm 2^j}_{j \in
    \Z}$. Observe that $\setprop{\supp \hat{g}_j}{j \in J}$ covers the
  frequency line $\R$. However, in addition, we also need to control
  in which order this covering is done.  On $S_0$ the support of
  $\hat{g}_j$ is moved to the left with increasing $j \in \Z$, while
  $\supp \hat{g}_j$ on $S_1$ is moved to the right.  To handle $S_0$
  and $S_1$ in the same setup, we introduce auxiliary functions
  $\varphi_i$, $i=0,1$,  to allow for a change of how we cover the
  frequency set $S_i$ with $\supp
  \hat{g}_j$.
  If we consider knots $\xi_j^{(0)}=-2^{-j}$ and $\xi_j^{(1)}=2^{j-1}$ for $j \in
  \Z$ and define two bijective functions on $\Z$,
  $\varphi_0=\mathrm{id}$ and $\varphi_1=-\mathrm{id}$, then we have
  \begin{equation}\label{wav}
    \supp \hat{g}_j \subset \itvcc{\xi_{\varphi_{0}(j)}^{(0)}}{\xi_{\varphi_{0}(j)+1}^{(0)}} \cup \itvcc{\xi_{\varphi_{1}(j)}^{(1)}}{\xi_{\varphi_{1}(j)+1}^{(1)}}
  \end{equation}
  The key point is that, for both $i=0$ and $i=1$, the
  contribution of $\supp\hat{g}_j$ in $S_i$ can be written uniformly as   $\itvccs{\xi_{\varphi_{i}(j)}^{(i)}}{\xi_{\varphi_{i}(j)+1}^{(i)}}$.
  	
  For Gabor systems the situation is simpler. Consider the Gabor-like
  orthonormal basis $\setsmall{T_k g_j}_{j,k\in \Z}$, where $g_j=E_j
  g$, $j \in \Z$, with $g\in L^2(\R)$ defined by $\hat g =
  \chi_{\itvcc{0}{1}}$. Hence, we only need one set $S_0=\R$ with
  knots $\Z$. Here, $\varphi_0$ is the identity on $\Z$.

We remark that, in most applications, each $\varphi_i$ will
  be an affine map of the form $z\mapsto a z+b$, $a,b\in \Z$. The
  choice of the
  knots $\xi^{(i)}_k$ is usually not unique, but is simply chosen to match the support
  of $\hat{g}_j$.
\end{ex}

Motivated by the concrete cases in Example \ref{exa:setup-1} we now formulate
the general setup as follows:
\begin{enumerate}[I)]
\item Let $S \subset \R$ be an at most countable collection of
  disjoint intervals
  \[
  S= \bigcup_{i \in I} S_i,
  \]
  where $I \subset \Z$. We write $S_i= \itvoc{\alpha_i}{\beta_i}$ with
  the convention that $\alpha_i=-\infty$ if
  $S_i=\itvoc{-\infty}{\beta_i}$, $\beta_i=\infty$ if
  $S_i=\itvoc{\alpha_i}{\infty}$, and $\alpha_i=-\infty$ and
  $\beta_i=\infty$ if $S_i=\R$. We assume an ordering of
  $\set{S_i}_{i\in I}$ so that $\beta_i \le \alpha_j$ whenever $i <
  j$. For each $i \in I$ we consider a sequence of knots \begin{equation} \label{60926a}  
   \seqsmall{\xi^{(i)}_k}_{k \in \Z}
  \subset S_i,\end{equation}  such that
  \[
  \lim_{k \to -\infty} \xi^{(i)}_k=\alpha_i, \quad \lim_{k \to \infty}
  \xi^{(i)}_k=\beta_i, \quad \xi^{(i)}_k \le \xi^{(i)}_{k+1}, \; k \in
  \Z.
  \] \vspace*{-1.5em}
\label{enu:S}
\item Let $c_j>0$. For each $j\in J$ we take $g_j \in L^2(\R)$ such
  that $\hat{g}_j$ is a bounded, real function with compact support in a finite
  union of the sets $S_i$, $i \in I$. We further assume that
  $\{c_j^{-1/2}\hat{g}_j\}_{j\in J}$ are uniformly bounded, i.e.,
  $\sup_j \normsmall[\infty]{c_j^{-1/2}\hat{g}_j}<\infty$.
\item For each $j \in J$, define the index set $I_j$ by
  \[
  I_j = \setprop{i \in I}{\hat{g}_j \neq 0 \text{ on } S_i}.
  \]
  On the other hand, for each $i \in I$, we fix an index set $J_i
  \subset J$ such that
  \[
  \setprop{j \in J}{\hat{g}_j \neq 0 \text{ on $S_i$}} \subset J_i .
  \]
  Note that $i \in I_j$ implies that $j\in J_i$.  Assume further that there is a bijective mapping
  $\varphi_i:J_i \to \Z$ such that, for each $j \in J$,
  \begin{equation}
  \supp \hat{g}_j \subset \bigcup_{i \in I_j}
  \itvcc{\xi^{(i)}_{\varphi_i(j)}}{\xi^{(i)}_{\varphi_i(j)+N}}\label{eq:supp-hat-gj}
\end{equation}
  for some $N \in \N$. Often, we take $J_i$ to be equal
  to the index set of the ``active'' generators $g_j$ on the interval
  $S_i$, that is, $J_i=\setprop{j \in J}{\hat{g}_j \neq 0 \text{ on
      $S_i$}}$, but this need not be the case, e.g., if $\setprop{j \in J}{\hat{g}_j \neq 0 \text{ on
      $S_i$}}$ is a finite set.
\label{item:hat-gj-support-assump}
\end{enumerate}

In the final step of our setup we have only left to define the dual
generators.
\begin{enumerate}[I)]
\setcounter{enumi}{3}
\item
 Let $h_j \in \check{L}^2(S)$ be given
by
\begin{equation}
\hat{h}_j(\gamma) =
\begin{cases}
  \sum_{n = -N+1}^{N-1} a^{(i)}_{\varphi_i(j),n} \;
  \hat{g}_{\varphi_i^{-1}(\varphi_i(j)+n)}(\gamma) \qquad &\text{for }
  \gamma \in S_i, i \in I_j \\
 0 & \text{for }
  \gamma \in S_i, i \notin I_j \\
\end{cases}
\label{eq:def-hj}
\end{equation}
where $\{ a^{(i)}_{k,n}\}_{i\in I,k \in \Z, n \in
\set{-N+1,\dots,N-1}}$ will be specified later.
\label{enu:h_j}
\end{enumerate}

\setcounter{ex}{2}
\begin{ex}[continuation]
  For a rigorous introduction of $\varphi_i$ in
  Example~\ref{exa:setup-1} above, we need to specify the sets $J_i$.
  In the wavelet case, as all dilates $D_{2^{-j}}\hat{\psi}$ have support intersecting \emph{both}
  $S_0$ and $S_1$, the active sets $J_i$, $i=0,1$, correspond to all
  dilations, that is, $\Z$. In the Gabor case, one easily also
  verifies that the active set $J_0$ is $\Z$.
\end{ex}

We are now ready to present the construction of dual GSI frames.  Recall from \eqref{60926a}
that the superscript $(i)$ on the points $\xi_k^{(i)}$ refer to the set $S_i.$  

\begin{thm}
\label{thm:GSI-dual-frames-1D-freq}
Assume the general setup \ref{enu:S}--\ref{enu:h_j}. Suppose that
\begin{equation}
\sum_{j \in J} c_j^{-1/2} \hat{g}_j(\gamma) =\chi_S(\gamma) \qquad
\text{for a.e. $\gamma \in \R$.}
\label{eq:part-unity-cond}
\end{equation}
Suppose further that $c_j\le 1/M_j$, where
\begin{equation}
M_j = \max\set{\xi_{\varphi_{\mathrm{max}I_j}(j)+2N}^{(\mathrm{max}I_j)}-\xi_{\varphi_{\mathrm{min}I_j}(j)}^{(\mathrm{min}I_j)}, \xi_{\varphi_{\mathrm{max}I_j}(j)+N}^{(\mathrm{max}I_j)}-\xi_{\varphi_{\mathrm{min}I_j}(j)-N}^{(\mathrm{min}I_j)} },\label{eq:def-M_j}
\end{equation}
and that $\{a^{(i)}_{k,n}\}_{k \in\Z, n\in\set{-N+1,\dots,N-1},i\in I}$ is a bounded sequence satisfying
\begin{equation}
  \label{eq:a_pl-c_j-cond}
  a^{(i)}_{\varphi_i(j),0} = 1  \quad \text{and} \quad
  \left(\frac{c_{\varphi_i^{-1}(\varphi_i(j)+n)}}{c_j}\right)^{1/2} a^{(i)}_{\varphi_i(j),n} +
  \left(\frac{c_{j}}{c_{\varphi_i^{-1}(\varphi_i(j)+n)}}\right)^{1/2}  a^{(i)}_{\varphi_i(j)+n,-n} = 2 ,
\end{equation}
for $n\in \set{1,\dots,N-1}$, $j\in J_i$, and $i \in I$. Then $\gsiZ$ and
$\gsiZ[h]$ are a pair of dual frames for $\check{L}^2(S)$.
\end{thm}
\begin{proof} First, note that for each $j\in J$, we have  $j\in J_{\mathrm{max}I_j}\cap J_{\mathrm{min}I_j}$ and therefore that $M_j$ in \eqref{eq:def-M_j} is
well-defined.

Now, note that by assumption~\eqref{eq:supp-hat-gj} and the definition
in~\eqref{eq:def-hj},
  \[
  \supp \hat{g}_j \subset \itvcc{\xi_{\varphi_{\mathrm{min}I_j}(j)}^{(\mathrm{min}I_j)}}{\xi_{\varphi_{\mathrm{max}I_j}(j)+N}^{(\mathrm{max}I_j)}}
  \]
  and
  \[
  \supp \hat{h}_j \subset \itvcc{\xi_{\varphi_{\mathrm{min}I_j}(j)-N}^{(\mathrm{min}I_j)}}{\xi_{\varphi_{\mathrm{max}I_j}(j)+2N}^{(\mathrm{max}I_j)}},
  \]
 where the constant $N$ is given by assumption \ref{item:hat-gj-support-assump}.
  Thus, if $j \in J$ and $0\neq m \in \Z$, then
  $\hat{g}_j(\gamma)\hat{h}_j(\gamma+c_j^{-1}m)=0$ for a.e.\ $\gamma\in\R$
  since $c_j^{-1}\ge M_j$ for each $j \in J$. Therefore, by
  Theorem~\ref{thm:HLW}, we only need to show that $\gsi$ and
  $\gsi[h]$ are Bessel sequences, satisfy the dual $\alpha$-LIC and that
  \begin{equation}
    \sum_{j \in {J}} c_j^{-1} \overline{ \hat{g}_j(\gamma)} \hat{h}_j(\gamma) = \chi_S(\gamma) \qquad \text{for a.e. } \gamma
    \in \R.\label{eq:-diag-eqn}
  \end{equation}
  holds.

  Choose $K>0$ so that $\frac{1}{\sqrt{c_j}} \abs{\hat{g}_j(\gamma)}\le K$
  for all $j \in J$ and a.e.\ $\gamma \in S$. For a.e. $\gamma \in S$
  we have
  \[
  \sum_{j\in J} \sum_{m \in \Z} \frac{1}{c_j} \absbig{
    \hat{g}_j(\gamma) \hat{g}_j(\gamma+c_j^{-1}m)} = \sum_{j\in J}
  \frac{1}{c_j} \absbig{\hat{g}_j(\gamma)}^2 \le K \sum_{j\in J}
  \frac{1}{c_j^{1/2}} \abs{\hat{g}_j(\gamma)} \le N K^2,
  \]
 where the last inequality follows from the fact that at most $N$
 functions from $\{\hat{g}_j:j \in J\}$ can be nonzero on a given
 interval $\itvccs{\xi^{(i)}_k}{\xi^{(i)}_{k+1}}$. It now follows from Theorem~\ref{thm:suff-cond-for-A-and-B-bound}(i)
  that $\gsi$ is a Bessel sequence in $\check{L}^2(S)$ with bound $NK^2$.

 A simple argument shows that if  $\{a^{(i)}_{k,n}\}_{k \in\Z,
   n\in\set{-N+1,\dots,N-1},i\in I}$ is a bounded sequence, then each
 term on the left-hand side of \eqref{eq:a_pl-c_j-cond} must be
 bounded with respect to $i\in I, j\in J_i$ and
 $n=-N+1,\cdots,N-1$. Let $M>0$ be such a bound. Then,
 for each $i\in I$ and $\gamma \in S_i$, we have
    \begin{align*}
    c_j^{-1/2}|\hat{h}_j (\gamma)|
    &\leq    \sum_{n = -N+1}^{N-1} |c_j^{-1/2}a^{(i)}_{\varphi_i(j),n} \; \hat{g}_{\varphi_i^{-1}(\varphi_i(j)+n)}(\gamma)|\\
    &\leq K  \sum_{n = -N+1}^{N-1}
    \abs{\frac{c_{\varphi_i^{-1}(\varphi_i(j)+n)}}{c_j}}^{1/2} \abs{a^{(i)}_{\varphi_i(j),n}}
    \leq KM(2N-1)=:L.
        \end{align*}
  Hence, the sequence of functions
  $\{c_j^{-1/2}\hat{h}_j\}_{j\in J}$ is  uniformly bounded by $L$.

For the remainder of the proof, we let $i \in I$ and $j \in J_i$ be fixed, but arbitrary.
 Note that the functions $\hat{g}_{\varphi_i^{-1}(\varphi_i(j)+n)}$, $n=0,1,\dots, N-1$,
  are the only nonzero generators $\set{\hat{g}_k}_{k \in J}$ on
  $\itvccs{\xi^{(i)}_{\varphi_i(j)+N-1}}{\xi^{(i)}_{\varphi_i(j)+N}}$.
 Hence, for $\hat{h}_{\varphi_i^{-1}(\varphi_i(j)+l)}$ only
 $l=-N+1,\cdots,2N-2$ can be nonzero on
 $\itvccs{\xi^{(i)}_{\varphi_i(j)+N-1}}{\xi^{(i)}_{\varphi_i(j)+N}}$. Thus,
 for
 $\gamma \in \itvccs{\xi^{(i)}_{\varphi_i(j)+N-1}}{\xi^{(i)}_{\varphi_i(j)+N}}$, we have
  \begin{align*}
  \sum_{j\in J}\sum_{m\in\Z}c_k^{-1}\absbig{\hat{h}_j (\gamma)\hat{h}_j (\ga+c_j^{-1}m)}&=
\sum_{j\in J}\sum_{m=-1}^1 c_j^{-1}\absbig{\hat{h}_j (\gamma)\hat{h}_j (\ga+c_j^{-1}m)}\\
&\leq 3L\sum_{j\in J}c_j^{-1/2}\absbig{\hat{h}_j (\gamma)}\\
  &=3L\sum_{\ell=-N+1}^{2N-2}
  c_{\varphi_i^{-1}(\varphi_i(j)+\ell)}^{-1/2}
  \absbig{\hat{h}_{\varphi_i^{-1}(\varphi_i(j)+\ell)}(\gamma)}\\
  &\leq 3L^2(3N-2).
    \end{align*}
It follows from Theorem~\ref{thm:suff-cond-for-A-and-B-bound}(i)
  that $\gsi[h]$ is also a Bessel sequence.
  Similar computations show that the $\alpha$-LIC holds:
  \begin{align*}
    \sum_{j\in J} \sum_{m \in \Z} \frac{1}{c_j} \absbig{
      \hat{g}_j(\gamma) \hat{h}_j(\gamma+c_j^{-1}m)} &= \sum_{j\in J}
    \frac{1}{c_j} \absbig{\hat{g}_j(\gamma)\hat{h}_j(\gamma)} \\ &\le
    \left(\sum_{j\in J} \frac{1}{c_j}
      \absbig{\hat{g}_j(\gamma)}^2\right)^{1/2} \left(\sum_{j\in J}
      \frac{1}{c_j} \absbig{\hat{h}_j(\gamma)}^2\right)^{1/2} \\ &\le
    N^{1/2}K B^{1/2},
  \end{align*}
  where $B$ is a Bessel bound for $\gsi[h]$.

  Finally, we need to show that \eqref{eq:-diag-eqn} holds.
 Set
  $\tilde{g}_k=c_k^{-1/2}g_k$, $k \in J$, and
  $\ell_n=\varphi_i^{-1}(\varphi_i(j)+n)$. Then, for a.e. $\gamma\in
  \itvccs{\xi^{(i)}_{\varphi_i(j)}}{\xi^{(i)}_{\varphi_i(j)+1}}$.
  \begin{align*}
    1&=\left(\sum_{k\in J}\hat{\tilde{g}}_k(\gamma)\right)^2=\left(\sum_{n=0}^{N-1} \hat{\tilde{g}}_{\ell_{n}}(\gamma)\right)^2 \\
    &= \hat{\tilde{g}}_{\ell_0}(\gamma) \left[ \hat{\tilde{g}}_{\ell_0}(\gamma) + 2\hat{\tilde{g}}_{\ell_{1}}(\gamma) + 2\hat{\tilde{g}}_{\ell_{2}}(\gamma) + \dots + 2\hat{\tilde{g}}_{\ell_{N-1}}(\gamma)\right] \\
    &+ \hat{\tilde{g}}_{\ell_{1}}(\gamma) \left[\phantom{\hat{\tilde{g}}_{\ell_0}(\gamma)+}\! \hat{\tilde{g}}_{\ell_{1}}(\gamma) + 2\hat{\tilde{g}}_{\ell_{2}}(\gamma) +\dots + 2\hat{\tilde{g}}_{\ell_{N-1}}(\gamma)\right] \\
    &+ \cdots \\
    &+ \hat{\tilde{g}}_{\ell_{N-1}}(\gamma) \left[\phantom{\hat{\tilde{g}}_{\ell_0}(\gamma)+ \hat{\tilde{g}}_{\ell_{1}}(\gamma)+2\hat{\tilde{g}}_{\ell_{2}}(\gamma) +}\quad \hat{\tilde{g}}_{\ell_{N-1}}(\gamma)\right].
  \end{align*}
  Clearly, each mixed term in this sum has coefficient $2$, e.g., $2\hat{\tilde{g}}_{\ell_n}(\gamma)\hat{\tilde{g}}_{\ell_m}(\gamma)$ whenever $n\neq
  m$. Replacing $\hat{\tilde{g}}_{\ell_n}$ with $c_{\ell_n}^{-1/2}\hat{g}_{\ell_n}$  yields a mixed
  term with coefficient $2c_{\ell_n}^{-1/2}c_{{\ell_m}}^{-1/2}\hat{g}_{\ell_n}(\gamma)\hat{g}_{\ell_m}(\gamma)$.
 By \eqref{eq:a_pl-c_j-cond}, we have
  \[ a_{\varphi_i(j),0}=1 \quad \text{and} \quad
  c_j^{-1} a^{(i)}_{\varphi_i(j),n} +
  c_{\varphi_i^{-1}(\varphi_i(j)+n)}^{-1}  a^{(i)}_{\varphi_i(j)+n,-n} = 2 c_j^{-1/2} c_{\varphi_i^{-1}(\varphi_i(j)+n)}^{-1/2},
 \]
  for $\ell=1,\dots,N-1, j\in J$. Hence, we can factor the sum
in the following way:
  \begin{align*}
    1&= c_{{\ell_0}}^{-1} \hat{g}_{\ell_0}(\gamma) \left[a_{{\ell_0},0}  \,\hat{g}_{\ell_0}(\gamma) +
      a_{{\ell_0},1} \,\hat{g}_{\ell_{1}}(\gamma) + \dots + 
    a_{{\ell_0},N-1} \,\hat{g}_{\ell_{N-1}}(\gamma)\right] \\
    &+ c_{\ell_{1}}^{-1} \hat{g}_{\ell_{1}}(\gamma) \left[
      a_{\ell_{1},-1}\,\hat{g}_{{\ell_0}}(\gamma)+ a_{\ell_{1},0}
      \,\hat{g}_{\ell_{1}}(\gamma) + \dots + a_{\ell_{1},N-2} \,\hat{g}_{\ell_{N-1}}(\gamma) \right] \\
    &+ \cdots \\
    &+ c_{\ell_{N-1}}^{-1} \hat{g}_{\ell_{N-1}}(\gamma)
    \left[a_{\ell_{N-1},-N+1} \,\hat{g}_{{\ell_0}}(\gamma) 
      + \dots + a_{\ell_{N-1},0}  \,\hat{g}_{\ell_{N-1}}(\gamma)\right] \\ &= \sum_{n=0}^{N-1} c_{\varphi_i^{-1}(\varphi_i(j)+n)}^{-1}
    \hat{g}_{\varphi_i^{-1}(\varphi_i(j)+n)}(\gamma)\hat{h}_{\varphi_i^{-1}(\varphi_i(j)+n)}(\gamma)= \sum_{k \in {J}} c_k^{-1} \hat{g}_k(\gamma) \hat{h}_k(\gamma)
  \end{align*}
  for a.e. $\gamma \in \itvccs{\xi^{(i)}_j}{\xi^{(i)}_{j+1}}$. Since $i \in
  I$ and $j\in J_i$ were arbitrary, the proof is complete.
\end{proof}

\begin{rem}
A few comments on the definition of $M_j$ are in place. Firstly, the only
feature of $M_j$ is to guarantee that $\hat{g}_j$ and $\hat{h}_j(\cdot
+ c^{-1}_jk)$ has non-overlapping supports for all $k \in \Z\setminus
\{0\}$.

Secondly, it is possible to make a sharper choice of $M_j$ in
Theorem~\ref{thm:GSI-dual-frames-1D-freq}. Indeed,
  let $n_{j,\mathrm{max}} =
\max\setpropsmall{n}{a^{(\mathrm{max}I_j)}_{\varphi_{\mathrm{max}I_j}(j),n} \neq 0}$ and let
$n_{j,\mathrm{min}}=\min\setpropsmall{n}{a^{(\mathrm{min}I_j)}_{\varphi_{\mathrm{min}I_j}(j),n}
  \neq 0}$. We then take:
\begin{equation}
M_j =
\max\set{\xi_{\varphi_{\mathrm{max}I_j}(j)+n_{j,\mathrm{max}}+N}^{(\mathrm{max}I_j)}-\xi_{\varphi_{\mathrm{min}I_j}(j)}^{(\mathrm{min}I_j)},
  \xi_{\varphi_{\mathrm{max}I_j}(j)+N}^{(\mathrm{max}I_j)}-\xi_{\varphi_{\mathrm{min}I_j}(j)+n_{j,\mathrm{min}}}^{(\mathrm{min}I_j)}
}.\label{eq:Mj-optimal}
\end{equation}
\end{rem}

For wavelet systems Theorem~\ref{thm:GSI-dual-frames-1D-freq} reduces to the construction from
\cite{MR2481501,Lemvig2012} of dual
wavelet frames in $L^2(\R)$.
\begin{cor}
  Let $a>1$ and $\psi \in L^2(\R)$. Suppose that $\hat\psi$ is  real-valued, that
  $\supp \hat\psi \subseteq\itvcc{-a^M}{-a^{M-N}} \cup \itvcc{a^{M-N}}{a^{M}}$ for some $M \in \Z, N \in
  \N$ and that
  \begin{equation}
\sum_{j \in \Z} \hat{\psi}(a^j\gamma) =1 \qquad \text{for a.e. \ $\gamma \in \R$}.\label{eq:part-unity-wavelets}
\end{equation}
Take $b_{n}\in \Bbb C$, $n=-N+1,\dots, N-1$, satisfying
\[
b_0 = 1 , \quad b_{-n}+b_{n} =2, \: n=1,2, \dots, N-1,
\]
and set $\ell:=\max\setprop{n}{b_n\neq 0}$. Let $b \in \itvoc{0}{a^{-M}(1+a^{\ell})^{-1}}$.
Then the function $\psi$ and the function $\tilde{\psi} \in L^2(\R)$
defined by
\begin{equation}\label{1404a}
\hat{\tilde{\psi}}(\gamma)= b \sum_{n=-N+1}^{N-1} b_n\, \hat \psi(a^{-n} \gamma) \qquad \text{for a.e. } \gamma\in \R
\end{equation}
generate dual frames $\{D_{a^j}T_{bk}\psi\}_{j,k\in\Z}$ and $\{D_{a^j}T_{bk}\tilde{\psi }\}_{j,k\in\Z}$ for $L^2(\R)$.
\end{cor}

\begin{proof}
Let $\phi=\sqrt{b}\psi$.
   We consider the wavelet system $\{D_{a^j}T_{bk}\phi\}_{j,k\in\Z}$ as a GSI system with
  $c_j=a^{j}b$ and $g_j=D_{a^{j}}\phi$ for $ j \in J=\Z$. We apply
  Theorem~\ref{thm:GSI-dual-frames-1D-freq} with $S_{0}=\itvoc{-\infty}{0}$ and
  $S_1=\itvoo{0}{+\infty}$, $\xi^{(0)}_j=-a^{M-j}$ and $\xi^{(1)}_j=a^{M-N+j}$ for $j \in \Z$. The
  assumption \eqref{eq:part-unity-cond} corresponds to
\[
\sum_{j \in \Z} \hat{\phi}(a^j\gamma) =\sqrt{b} \qquad \text{for } \gamma \in \R
\]
which is satisfied by \eqref{eq:part-unity-wavelets}.
Let $a^{(0)}_{j,n}=a^{-n/2}b_n$ and $a^{(1)}_{j,n}=a^{n/2}b_{-n}$ for  $n=-N+1,\dots, N-1$ and $j\in\Z$ and define $\varphi_0,\varphi_1:\Z\to\Z$ with  $\varphi_0(j)=j$ and $\varphi_1(j)=-j$.
The definition of $h_j$ in \eqref{eq:def-hj} reads
\[
\hat{h}_j(\gamma) =
\begin{cases}
  \sum_{n=-N+1}^{N-1} a^{(0)}_{j,n} \;
a^{(j+n)/2}\hat{\phi}(a^{(j+n)}\gamma) \qquad &\text{for }
  \gamma \in S_0, \\[2mm]
 \sum_{n=-N+1}^{N-1} a^{(1)}_{-j,-n} \;
a^{(j+n)/2}\hat{\phi}(a^{(j+n)}\gamma) & \text{for }
  \gamma \in S_1.
\end{cases}
\]
Setting $\hat{\tilde{\phi}}=D_{a^{-j}}h_j$ yields
\[
\hat{\tilde{\phi}}(\gamma) = \sum_{n=-N+1}^{N-1} b_{n} \;
\hat{\phi}(a^{n}\gamma)\qquad \gamma\in\R.
\]

For all $ j\in\mz$, we have $I_j=\{0,1\}$ and thus
$\mathrm{min}I_j=0$ and $\mathrm{max}I_j=1$. Note also that $\ell=n_{j,\mathrm{max}}= -
n_{j,\mathrm{min}}$ for all $j\in\mz$. Hence,
condition~\eqref{eq:Mj-optimal} reads
\begin{align*}
M_j &= \max\set{\xi_{-j+n_{j,\mathrm{max}}+N}^{(1)}-\xi_{j}^{(0)}, \xi_{-j+N}^{(1)}-\xi_{j+n_{j,\mathrm{min}}}^{(0)} }\\
&=\max\set{a^{M-j+n_{j,\mathrm{max}}}+a^{M-j}, a^{M-j}+a^{M-j-n_{\mathrm{j,min}}} }\\
&=a^{M-j}(a^{\ell}+1).
\end{align*}

From Theorem~\ref{thm:GSI-dual-frames-1D-freq} we
have that  $\{D_{a^j}T_{bk}\phi\}_{j,k\in\Z}$ and
$\{D_{a^j}T_{bk}\tilde\phi\}_{j,k\in\Z}$ are dual frames for $b \in \itvoc{0}{a^{-M}(1+a^{\ell})^{-1}}$. By setting
$\tilde\psi=\sqrt{b}\tilde{\phi}$, we therefore have that   $\{D_{a^j}T_{bk}\psi\}_{j,k\in\Z}$ and
$\{D_{a^j}T_{bk}\tilde\psi\}_{j,k\in\Z}$ are dual frames; it is clear that $\hat{\tilde\psi}$ is given
by the formula in \eqref{1404a}.
\end{proof}

Using the Fourier transform we can move the construction in
Theorem~\ref{thm:GSI-dual-frames-1D-freq} to the time domain. In this setting we obtain dual frames
 $\set{\modu[b_pm]g_p}_{m\in \Z, p\in \Z}$ and $\set{\modu[b_pm]h_p}_{m\in \Z, p\in \Z}$ with
 compactly supported generators $g_p$ and $h_p$. A simplified version of this result, useful for
 application in Gabor analysis, is as follows.

\begin{thm}
\label{thm:GSI-dual-frames-1D}
Let $\setprop{x_p}{p \in \Z} \subset \R$ be a sequence such that
\[
\lim_{p \to \pm \infty}x_p = \pm \infty, \quad x_{p-1}\le x_p, \quad \text{and } x_{p+2N-1}-x_p \le M, \quad p
\in \Z,
\]
for some constants $N \in \N$ and $M>0$. Let $g_p \in L^2(\R)$,
$p\in\Z$, be real-valued functions with such that
$\{b_p^{-1/2}g_p\}_{p \in \Z}$
are uniformly bounded functions with $\supp g_p \subset
\itvcc{x_p}{x_{p+N}}$. Assume that $\sum_{p \in \Z}
b_p^{-1/2} g_p(x) =1$ for a.e.\ $x\in \R$. Let
\[
h_p(x) = \sum_{n=-N+1}^{N-1} a_{p,n} \; g_{p+n}(x),  \quad
\text{for } x\in \R,
\]
where $\{a_{p,n}\}_{p\in \Z,n\in \set{-N+1,\dots,N-1}}$ is a
bounded sequence in $\Bbb C$. Suppose that $0<b_p<1/M$ and that
\begin{equation}
  \label{eq:a_pl-b_p-cond}
  a_{p,0} = 1  \quad \text{and} \quad
  \left(\frac{b_{p+n}}{b_p}\right)^{1/2} a_{p,n} +
  \left(\frac{b_{p}}{b_{p+n}}\right)^{1/2}  a_{p+n,-n} = 2 ,
\end{equation}
for $p\in \Z$ and $n=1,\dots,N-1$. Then $\set{E_{b_p m} {g}_p}_{m \in
      \Z, p \in \Z}$ and
$\set{E_{b_p m} {h}_p}_{m \in
      \Z, p \in \Z}$ are a pair of dual frames for $L^2(\R)$.
\end{thm}

Theorem~\ref{thm:GSI-dual-frames-1D} generalizes results on SI systems by Christensen and
Sun~\cite{MR2407863} and Christensen and Kim~\cite{MR2587577} in the following way: Taking $b_p=b$
for all $p\in \Z$, Theorem~\ref{thm:GSI-dual-frames-1D} reduces to
\cite[Theorem 2.2]{MR3061703} and to \cite[Theorem~2.5]{MR2407863}
when further choosing $a_{p,n}=0$ for $n=1,\dots,N-1$ and $p \in \Z$ and to Theorem 3.1 in
\cite{MR2587577} when choosing $g_p=\tran[p]g$ for some $g \in
L^2(\R)$.



\section{Wave packet systems}
\label{sec:wave-packet-systems}

Let $b>1$, and let $\{(a_j, d_j)\}_{j\in J}$ be a countable set in
$\R^+\times \R$. The wave packet system $\wpj$ is a
GSI system with
 \[
  g_j=D_{a_j}E_{d_j}\psi \; \text{and} \; c_j=a_jb, \quad j\in
  J.
\]
It is of course possible to apply the dilation, modulation and
translation operator in a different order than in $\wpj$. Indeed, we
will also consider the collection of functions $\siwp$, which is a
shift-invariant version of the wave packet system. It takes the form of
a GSI system with
 \[
  g_j=D_{a_j}E_{d_j}\psi  \; \text{and} \;  c_j=b, \quad j\in J,
\]
and it contains the Gabor-like system $\set{T_{bk} E_{aj} \psi}_{j,k\in \Z}$, but not the wavelet system, as a
special case.

There are four more ordering of the  dilation, modulation and
translation operators; however, the study of these systems reduces to
either the study of $\wpj$ or $\siwp$. This is clear from the commutator relations:
\[
D_{a_j}T_{bk}E_{d_j} = \myexp{- 2\pi i d_jbk} D_{a_j}E_{d_j}
 T_{bk} = \myexp{- 2\pi i d_jbk} E_{a_j^{-1} d_j} D_{a_j} T_{bk}
\]
and
\[
T_{bk}D_{a_j}E_{d_j} = T_{bk} E_{a_j^{-1} d_j} D_{a_j} = \myexp{- 2\pi i a_j^{-1}d_jbk} E_{a_j^{-1} d_j} T_{bk}  D_{a_j}.
\]

The following result is a special case of
Theorem~\ref{thm:suff-cond-for-A-and-B-bound} for the case of wave packet
systems of the form $\wpj$; a direct proof was given in \cite{MR2420867}.
\begin{thm}
  \label{thm:ole-rahimi}
Let $b>1$, and let $\{(a_j, d_j)\}_{j\in J}$ be a countable set in
$\R\setminus\{0\}\times \R$. Assume that
  \begin{equation}\label{eq:bzm1}
    B:=\frac{1}{b}\sup_{\gamma\in\R}\sum_{j\in J}\sum_{k\bbz}|\hat{\psi}(a_j\gamma
    -d_j)\hat{\psi}(a_j\gamma-d_j- k/b)|<\infty.
  \end{equation}
  Then $\wpj$ is a Bessel sequence with bound $B$. Further, if also
  \[
    A:=\frac{1}{b}\inf_{\gamma\in\R}\bigg(\sum_{j\in J}|
    \hat \psi (a_j\gamma-d_j)|^{2}
    - \sum_{j\in J}\sum_{0\neq k\bbz} |\hat \psi
    (a_j\gamma-d_j)\hat \psi (a_j\gamma-d_j- k/b)|\bigg)>0,
  \]
  then $\wpj$ is a frame for $\ltr$ with bounds $A$ and $B$.
\end{thm}

A similar result holds for the shift-invariant wave packet system $\siwp$.

\subsection[LIC and a lower bound for the Calder\'on sum]{Local integrability conditions and a lower bound for the Calder\'on sum}
\label{sec:local-integr-cond}

Theorem~\ref{thm:ole-rahimi} allows us to construct frames, even tight frames, without worrying about
technical local integrability conditions; in fact, the condition~\eqref{eq:bzm1}
implies that the $\alpha$-LIC is satisfied.
\begin{lem}[\!\!\cite{JakobsenReproducing2014}]
\label{lem:CC-implies-aLIC}
  If \eqref{eq:bzm1} holds, then the $\alpha$-LIC for wave packet
  systems holds with respect to some set $E \in \cE$, i.e.,
\[
L'(f)= \sum_{j\in J} \sum_{k\bbz} \int_{\R}
\abssmall{\hat{f}(\gamma) \hat{f}(\gamma-k/(a_jb)) \hat{\psi}(a_j\gamma
    -d_j)\hat{\psi}(a_j\gamma-d_j- k/b)} \, \,\mathrm{d}\gamma <\infty
\]
for all $f \in \mathcal{D}_E$.
\end{lem}

On the other hand, as we will see in Example~\ref{exa:joachim},
condition \eqref{eq:bzm1} does not imply the
LIC. Indeed, for wave packet system, if possible, one should work with $\alpha$-LIC instead of the
LIC. We will see further results supporting this claim in the following.  We continue our study of
local integrability conditions with a special case of the wave packet system $\wpj$ that is highly
redundant. Given $a>1$ and a sequence $\{d_m\}_{m\bbz}$ in $\R$, we consider the collection of
functions
\begin{equation}
\label{eq:50301c}
\{D_{a^j}T_{bk}E_{d_m}\psi\}_{j,m,k\bbz}.
\end{equation}
which can be considered as GSI systems with $c_{(j,m)}=a^jb$ and $g_{j,m}=D_{a^j}E_{d_m}\psi$. For
wave packet systems of this form the point set $\{(a_j,d_j)\}_{j \in J} \subset
\R^+ \times \, \R$ is a separable set of the form $\{(a^j,d_m)\}_{j,m \in Z}$. For each
fixed $m \in \Z$ in~\eqref{eq:50301c} the system $\{D_{a^j}T_{bk}E_{d_m}\psi\}_{j,k\in\Z}$ is a
wavelet system (with generator $E_{d_m}\psi$).

The first obvious constraint is that if the system $\wps$ is a frame, then the sequence
$\{d_m\}_{m\in\Z}$ cannot be bounded. In fact, in that case the sequence $\{d_m\}_{m\in\Z}$ has an
accumulation point and thus $\wps$ cannot be a Bessel sequence, see \cite[Lemma 2.3]{MR2420867}.
We will now show that a wave packet system on the form~\eqref{eq:50301c} cannot satisfy
condition~\eqref{eq:loc}.

\begin{lem}
\label{lem:not-loc}
Assume that  $\{d_m\}_{m\in\Z}$ is unbounded and $\psi\neq 0$. Then
\begin{equation}\label{eq:mycondition}
\sum_{\{j:a^jb>M\}}\sum_{m\bbz}|a^{\frac{j}{2}}\hat{\psi}(a^j\cdot-d_m)|^2\notin \Lloc(\R\setminus E),
\end{equation}
for all $M >0$ and all  $E \in \cE$.
\end{lem}
\begin{proof}
Let $M >0$ and let $E \in \cE$. Assume that $\{d_m\}_{m\in\Z}$ is not bounded below.
Note that $\{j:a^jb>M\}=\{j: j\geq M'\}$ where $M'=[\frac{\ln M-\ln b}{\ln a}]+1$.
 Taking $K=[1,a]\setminus E$, we have
  \begin{align}\label{eq:0.9}
    I_{M}&:=
   \sum_{m\bbz} \sum_{j=M'}^\infty\int_Ka^j| \hat{\psi}(a^j\gamma+d_m)|^2d\ga=\sum_{m\bbz}\sum_{j=M'}^\infty\int_{a^jK}| \hat{\psi}(\ga+d_m)|^2d\ga\nonumber\\
&=\sum_{m\bbz}\sum_{j=M'}^\infty\int_{a^j[1,a]}|\hat{\psi}(\ga+d_m)|^2d\ga\geq
\sum_{m\bbz}\int_{a^{M'}}^\infty| \hat{\psi}(\ga+d_m)|^2d\ga .
  \end{align}
  Since $\{d_m\}_{m\in\Z}$ is not bounded below, there exists a subsequence
  $\{d_{m_l}\}_{l=1}^\infty$ of $\{d_m\}_{m\in\Z}$ such that $d_{m_l}\rightarrow-\infty$. Hence, for each $N\in
  \R$ there exists $L\in\Bbb N$ such that $d_{m_l}<N-a^{M'}$ for $l\geq L$. Thus
  \[
  \int_{a^{M'}+d_{m_l}}^\infty|\hat{\psi}(\gamma)|^2>\int_{N}^\infty|\hat{\psi}(\gamma)|^2\,\mathrm{d}\gamma,
  \]
  for all $l\geq L$. But in this case, using inequality~\eqref{eq:0.9} and choosing $N$ small enough, we have
  \[
  I_{M}\geq\sum_{l\geq L}\int_{N}^\infty|\hat{\psi}(\gamma)|^2\,\mathrm{d}\gamma=\infty.
  \]
 A similar argument  shows that if $\{d_m\}_{m\in\Z}$ is not bounded above, then $I_M=\infty$.
Therefore \eqref{eq:mycondition} holds.
\end{proof}

Thus, for the case where $\{d_m\}_{m\in\Z}$ is
unbounded, it is impossible for a wave packet system $\wps$ to satisfy
the LIC-condition.  On the other hand,
if $\{d_m\}_{m\in\Z}$ is bounded, we know that it is impossible for such a system to form a Bessel
sequence. Hence, the wave packet system $\wps$ cannot simultaneously be a Bessel sequence and
satisfy the LIC.
\begin{cor}
\label{thm:wps-bessel-not-lic}
If the system $\wps$ is a Bessel sequence, then this system does not satisfies the LIC.
\end{cor}

The following example
introduces a family of wave packet tight frames of the form 
  \[ \{D_{a^{j}}T_{bk}E_{d_m}\psi\}_{j\bbz,k\bbz, m
  \bbz\setminus\{0\}} \] that satisfies \eqref{eq:bzm1} and thus the $\alpha$-LIC
by Lemma~\ref{lem:CC-implies-aLIC}, but not the LIC.

\begin{ex}
\label{exa:joachim}
Let $\psi \in L^2(\R)$ be a Shannon-type
scaling function defined by
$\hat{\psi}=\chi_{\itvcc{-1/4}{1/4}}$. Let $a=2$, $b=1$, and define $d_m= \sign{m}(2^{|m|}-\frac{3}{4})$ for $m\in\Z\setminus \{0\}$.
We will first argue that the wave packet system
$\{D_{a^{j}}T_{bk}E_{d_m}\psi\}_{j\bbz,k\bbz, m \bbz\setminus\{0\}}$ is a tight frame for $\ltr$ with bound $1$. To see this, we will simply verify the  conditions in
Theorem~\ref{thm:ole-rahimi}. All we need to do is to prove that
\begin{equation}\label{p}
\sum_{j\bbz} \sum_{m\bbz\setminus \{0\}}| \hat \psi (2^j
\gamma-d_m)|^{2} =1 \quad \text{for a.e. $\gamma \in \R$}.
\end{equation}
With our definition of $\psi$, this amounts to verifying that the sets
\begin{equation}
  \label{eq:1a} I_{m,j}:=2^{-j} \itvco{ -\tfrac{1}{4}+d_m}{\tfrac{1}{4}+d_m}, \ j\bbz, m\bbz
  \setminus \{0\},
\end{equation}
form a disjoint covering of $\R$. To see this, let $m\in\N$ and $k\in\Z$, then
\[
I_{m,m-k}=2^{-m+k}\itvco{-1+2^{m}}{-\tfrac{1}{2}+2^{m}}=2^{k}\Bigl(1+2^{-m}\itvco{-1}{-\tfrac{1}{2}}\Bigr).
\]
The sets $1+2^{-m}\itvco{-1}{-\frac{1}{2}}$, $m\in\N$, form a disjoint
covering of $\itvco{\frac{1}{2}}{1}$. Hence the sets $2^{k}(1+2^{-m}\itvco{-1}{-\frac{1}{2}})$, $k\in\Z, m\in\N$, form a disjoint covering of $\itvoo{0}{\infty}$.

A similar argument for $m<0$ shows that
$\{I_{m,j}\}_{j\in\Z,m\in -\N}$ is a disjoint covering of
 $\itvoo{-\infty}{0}$.
 We conclude that \eqref{p} holds. Thus, 
   the system $\{D_{2^{j}}T_{k}E_{d_{m}}\psi\}_{j\bbz,k\bbz, m
     \bbz\setminus\{0\}}$ is a tight frame.

  Since $\{D_{2^{j}}T_{k}E_{d_{m}}\psi\}_{j\bbz,k\bbz, m \bbz\setminus\{0\}}$ satisfies
  \eqref{eq:bzm1}, it satisfies the $\alpha$-LIC by Lemma~\ref{lem:CC-implies-aLIC}. On the other hand, by
  Corollary~\ref{thm:wps-bessel-not-lic} the wave packet system does
  not satisfy the LIC.
\end{ex}

We finally consider lower bounds of the Calder\'on sum for wave packet
systems.  From Corollary~\ref{thm:lic-lower-bound} we know that any
GSI frame that satisfies the LIC will have a lower bounded Calder\'on
sum. A first result utilizing this observation is the following result
for the shift-invariant version of the  wave packet systems.

\begin{thm}
Let $\psi\in \ltr$. If $\siwp$ is a frame with lower bound $A$, then
\begin{equation}
\label{eq:0423c}
A\leq \sum_{j\in J} \frac{a_j}{b}|\hpi(a_j\gamma-d_j)|^2 \qquad \text{a.e.} \, \gamma\in \R.
\end{equation}
\end{thm}
\begin{proof}
  Any SI system satisfies the LIC. Hence, in particular, the system $\siwp$ satisfies the LIC. The result is now immediate
  from Corollary~\ref{thm:lic-lower-bound}.
\end{proof}

On the other hand, Example~\ref{exa:joachim} shows that, in general,
we cannot expect wave packet frames $\wpj$, even tight frames, to satisfy
the LIC. Hence, in general, we can only say that if the wave packet
system $\wpj$ is a frame with lower bound $A$ \emph{and} if it satisfies the LIC, then
\begin{equation}
\label{eq:0423c}
A\leq \sum_{j\in J} \frac{1}{b}|\hpi(a_j\gamma-d_j)|^2 \qquad \text{a.e.} \, \gamma\in \R.
\end{equation}

\subsection{Constructing dual wave packet frames}
\label{sec:constr-dual-wp-frames}

We now want to apply the general construction of dual GSI frames in Theorem~\ref{thm:GSI-dual-frames-1D-freq} to
the case of wave packet systems. We first consider how to construct
suitable partitions of the unity.

\begin{ex}
\label{exa:part-of-unity}
Let  $f:[a^{j_0},a^{j_0+1}]\rightarrow\R$ be a continuous function
such that $f(a^{j_0})=0$, $f(a^{j_0+1})=1$ and
$f(-\ga+a^{j_0+1}+a^{j_0})+f(\gamma)=1$. Define
\[
\hpi(\gamma):=\begin{cases}
                       f(|\ga|)    &     a.e.\ |\ga| \in [a^{j_0}, a^{j_0+1}],\\
                       1-f(\frac{|\gamma|}{a})      & a.e.\ |\ga| \in [a^{j_0+1},a^{j_0+2}],\\
                                                        0     &
                                                       \text{otherwise.}\end{cases}
\]
For  almost every $\gamma \in [-a^{j_0+1},a^{j_0+1}]$ there is exactly one $j\in J$ such that
$a^j|\gamma |\in [a^{j_0},a^{j_0+1}]$. Hence $\hpi (a^j\gamma)=f(|a^j\ga|),\,\hpi(a^{j+1}\gamma)=1-f(\frac{a^{j+1}|\gamma|}{a})$. Then for $J=\N\cup \{0\}$, we have
\[
\sum_{j\in J} \hpi (a^j\gamma)=\begin{cases}1& a.e.\ \gamma \in
  [-a^{j_0+1},a^{j_0+1}], \\
1-f(\frac{|\gamma|}{a})& a.e.\ |\ga|\in [a^{j_0+1},a^{j_0+2}], \\
0&\text{otherwise.}\end{cases}
\]
By shifting this function along $d\Z$, where $d=a^{j_0+1}(a+1)$, we have
\[
\sum_{m\in\Z}\sum_{j\in J}\hpi
(a^j(\ga-dm))=\sum_{m\in\Z}\sum_{j\in J}\hpi (a^j\ga-a^jdm)=1 \quad
a.e.\ \gamma \in \R .
 \]
\end{ex}

In the remainder of this section we let $\psi$ be defined as in
Example~\ref{exa:part-of-unity}. Note that depending on the choice of $f$, we can make $\hat{\psi}$
as smooth as we like. Hence, we can construct generators $\psi$ that
are band-limited functions with
arbitrarily fast decay in time domain and that have a partition of
unity property.

Note that if $ f(\ga)\in [0,1] $ for all $ \ga\in\itvcc{a^{j_0}}{a^{j_0+1}} $, the function $ \hpi $ is non-negative. In this case we can use the partition of unity to construct tight frames. Indeed, for the parameter choice
 $b< 2a^{j_0+2}$ the sums over $k \in \Z$ in
 Theorem~\ref{thm:ole-rahimi} only have non-zero terms for $k=0$. Define
 $\hat{\phi}=\hpi^{1/2}$. Since
\[
\sum_{m\in\Z}\sum_{j\in J}|\hat{\phi} (a^j\ga-a^jdm)|^2=1,
\]
it follows from Theorem~\ref{thm:ole-rahimi} that
$\{D_{a^j}T_{bk}E_{a^jdm}\phi\}_{k,m\in\Z,j\in J}$ is a 1-tight wave
packet frame for $\ltr$.

However, taking the square root of $\hpi$ might destroy desirable
properties of the generator, e.g., if $f$ is a polynomial, then $\hpi$ is piecewise polynomial,
but this property is not necessarily inherited by $\hat{\phi}:=\hpi^{1/2}$.
By constructing dual frames from
Theorem~\ref{thm:GSI-dual-frames-1D-freq}, we can circumvent this issue.

\begin{ex}
\label{exa:dual-wp-frames}
In order to apply Theorem~\ref{thm:GSI-dual-frames-1D-freq} we need to
setup  notation.
Let $\tilde{\psi}=\sqrt{b}\psi$. Consider the wave packet system  $\{D_{a^j}T_{bk}E_{a^jdm}\tilde{\psi}\}_{k,m\in\Z,j\in \N_0}$ as a GSI-system $\{T_{c_{(j,m)}k}g_{(j,m)}\}_{j\in\N_0,m,k\in\Z}$, where $g_{(j,m)}=D_{a^j}E_{a^jdm}\tilde{\psi}$ and $c_{(j,m)}=a^jb$ for all $j\in\N_0$ and $m\in\Z$.

For $i\in I$, define $S_i=\itvoc{di}{ d(i+1)}$.  Since $\hat{g}_{(j,m)}=D_{a^{-j}}T_{a^jdm}\hat{\tilde{\psi}}$, we have
\[
I_{(j,m)}=\set{m,m-1},\quad J_i=\setprop{(j,i),(j,i+1)}{j \in \N_0}.
\]
Define the knots $\seqsmall{\xi_{k}^{(i)}}_{k,i \in \Z}$ by
\begin{equation*}
\xi_{k}^{(i)}=
\begin{cases}
  -a^{j_0+3-k}+d(i+1) \qquad &k> 0, \\
 a^{j_0+k}+di & k\leq 0.\\
\end{cases}
\label{eq:xi}
\end{equation*}
Then the sequence $\seqsmall{\xi_{k}^{(i)}}_{k,i \in \Z}$ fulfills the
properties in \ref{enu:S}
in the setup from Section \ref{sec:constr-dual-gsi-frames}. For each $i\in \Z$, we define the bijective mapping $\varphi_i : J_i\mapsto \Z$ by
$\varphi_i (j,i)=-j$ and $\varphi_i (j,i+1)=1+j$ for all $j\in
\N_0$. Then  
\[
 \supp \hat{g}_{(j,m)} \subset \bigcup_{i \in I_{(j,m)}}
 \itvcc{\xi^{(i)}_{\varphi_i(j,m)}}{\xi^{(i)}_{\varphi_i(j,m)+2}}
 \]
The definition of $h_{(j,m)}$ for $j \in \N$ and $m\in \Z$, is
\begin{equation*}
\hat{h}_{(j,m)}=
\begin{cases}
\sum_{n=-1}^1 a^m_{-j,n}\hat{g}_{(j-n,m)}(\gamma) \qquad &\gamma \in S_m, \\
  \sum_{n=-1}^1 a^{m-1}_{j+1,n}\hat{g}_{(j+n,m-1)}(\gamma) \qquad &\gamma \in S_{m-1},\\
0 \qquad& \gamma \in \R\setminus (S_m \cup S_{m-1}),
\end{cases}
\end{equation*}
and, for $j=0$, we have
\begin{equation*}
\hat{h}_{(0,m)}=
\begin{cases} a^m_{0,-1}\hat{g}_{(1,m)}(\gamma)+\hat{g}_{(0,m)}(\gamma)+a^m_{0,1}\hat{g}_{(0,m+1)}(\gamma) \qquad &\gamma \in S_m, \\ a^{m-1}_{1,-1}\hat{g}_{(0,m-1)}(\gamma)+\hat{g}_{(0,m)}(\gamma)+a^{m-1}_{1,1}\hat{g}_{(1,m)}(\gamma) \qquad &\gamma \in S_{m-1},\\
0 \qquad& \gamma \in \R\setminus (S_m \cup S_{m-1}).
\end{cases}
\end{equation*}
Hence, if we define $\hat{\phi}_1=T_{-a^jdm}D_{a^j}\hat{h}_{(j,m)}$
and $\hat{\phi}_0=T_{-dm}\hat{h}_{(0,m)}$. Take any two numbers $b_1,
b_{-1}$ in $\R$ with $b_1+b_{-1}=2$ and set
$a^{m}_{k+1,-n}=a^{m}_{-k,n}=a^{n/2}b_n$ for all $k\in \N$ and $m\in
\Z$. We then have
 \begin{equation}
\hat{\phi}_1(\gamma)= b_{-1}\hat{\tilde{\psi}}(a\gamma)+
\hat{\tilde{\psi}}(\gamma) + b_{1}\hat{\tilde{\psi}}(a^{-1}\gamma)
\qquad \gamma\in \R .\label{eq:phi1}
\end{equation}
For $j=0$, based on Theorem~\ref{thm:GSI-dual-frames-1D-freq}, we should set $a^m_{0,-1}=a^m_{1,1}=a^{-1/2}b_{-1}$, also we set $c_{-1}=a^m_{1,-1}$ and $c_1=a^m_{0,1}$ for all $m\in \Z$, where $c_{-1},c_{1}\in \R$ and $c_{-1}+c_{1}=2$. In this case, we have
\begin{equation}
\hat{\phi}_0= b_{-1}\hat{\tilde{\psi}}(a\gamma)+
\hat{\tilde{\psi}}(\gamma) + c_1\hat{\tilde{\psi}}(\gamma-d)+
c_{-1}\hat{\tilde{\psi}}(\gamma+d) \qquad \gamma\in \R .
\label{eq:phi0}
\end{equation}
By Theorem~\ref{thm:GSI-dual-frames-1D-freq}, we conclude that the systems
\[
\{T_{bk}E_{a^jdm}\psi\}_{k,m\in\Z} \cup
\{D_{a^j}T_{bk}E_{a^jdm}\psi\}_{k,m\in\Z,j\in\N}
\]
and
\[
\{b^{1/2} \, T_{bk}E_{dm}\phi_0\}_{k,m\in\Z} \cup
\{b^{1/2} \, D_{a^j}T_{bk}E_{a^jdm}\phi_1\}_{k,m\in\Z,j\in\N}
\]
are dual wave packet frames for $L^2(\R)$. Note that in the definitions
of $\phi_0$ and $\phi_1$ in \eqref{eq:phi0} and \eqref{eq:phi1}, respectively, we are free to choose any set of coefficients satisfying
 $b_1+b_{-1}=2$ and $c_{-1}+c_{1}=2$.
\end{ex}

In Example~\ref{exa:dual-wp-frames} we constructed dual wave packet
frames with \emph{two} generators, akin to the case of scaling and
wavelet functions for non-homogeneous wavelet systems. By a special
choice of the coefficients  $b_1, b_{-1}, c_{-1}, c_{1}$, we can
reduce the number of generators to one.

\begin{ex}
Take $b_1=b_{-1}= c_{-1}= c_{1}=1$ in \eqref{eq:phi1} and
\eqref{eq:phi0}. Note that $\hat{\tilde{\psi}}(a^{-1}\gamma)$ is equal
to $\hat{\tilde{\psi}}(\gamma-d)+\hat{\tilde{\psi}}(\gamma+d)$ on the
support of $\hat{\psi}$. Thus, $\hat{\phi}_0$ and
$\hat{\phi}_1$ agree on $\supp\hat{\psi}$. Hence
 if we set
\[
\hat{\phi}(\gamma)=b\hpi(a\gamma)+b\hpi(\gamma)+b\hpi(\ga-d)+b\hpi(\ga+d),\quad
\gamma \in \R,
\]
then the wave packet systems $\{D_{a^j}T_{bk}E_{a^jdm}\psi\}_{k,m\in\Z,j\in J}$ and
$\{D_{a^j}T_{bk}E_{a^jdm}\phi\}_{k,m\in\Z,j\in J}$ are dual frames
for $L^2(\R)$ for $b<a^{-j_0}(2a^2+a-1)^{-1}$.
Alternatively, we can take
\[
\hat{\phi}(\gamma)=b\hpi(a\gamma)+b\hpi(\gamma)+b\hpi(a^{-1}\gamma),\quad
\gamma \in \R,
\]
in which case we need to take $b< a^{-j_0-2}(a+1)^{-1}$ to obtain dual frames.
\end{ex}

\noindent{\bf Acknowledgment:} Marzieh Hasannasab would like to thank Kharazmi University
for financial support during the preparation of the first draft of
this manuscript. The three authors thank the reviewers for useful suggestions that improved the presentation.

\def\cprime{$'$} \def\cprime{$'$} \def\cprime{$'$} \def\cprime{$'$}
  \def\uarc#1{\ifmmode{\lineskiplimit=0pt\oalign{$#1$\crcr
  \hidewidth\setbox0=\hbox{\lower1ex\hbox{{\rm\char"15}}}\dp0=0pt
  \box0\hidewidth}}\else{\lineskiplimit=0pt\oalign{#1\crcr
  \hidewidth\setbox0=\hbox{\lower1ex\hbox{{\rm\char"15}}}\dp0=0pt
  \box0\hidewidth}}\relax\fi} \def\cprime{$'$} \def\cprime{$'$}
  \def\cprime{$'$} \def\cprime{$'$} \def\cprime{$'$} \def\cprime{$'$}

\end{document}